\documentclass[12pt, reqno]{amsart}
\usepackage{amssymb,amsmath,amsthm,cancel,hyperref,color, enumerate}
\usepackage[pdftex]{graphicx}
\hypersetup{colorlinks=true,linkcolor=blue,citecolor=magenta}
\newcommand{\Z}[0]{\mathbb{Z}}

\newcommand{\rc}[1]{\frac{1}{#1}}

\newcommand{\subeq}[0]{\subseteq}
\newcommand{\supeq}[0]{\supseteq}

\newcommand{\al}[0]{\alpha}
\newcommand{\be}[0]{\beta}
\newcommand{\ga}[0]{\gamma}
\newcommand{\Ga}[0]{\Gamma}
\newcommand{\de}[0]{\delta}
\newcommand{\De}[0]{\Delta}

\newcommand{\la}[0]{\lambda}

\newcommand{\om}[0]{\omega}

\newcommand{\ba}[1]{\left[ {#1} \right]}
\newcommand{\bc}[1]{\left\{ {#1} \right\}}
\newcommand{\pa}[1]{\left( {#1} \right)}

\newcommand{\fl}[1]{\left\lfloor {#1}\right\rfloor}
\newcommand{\ce}[1]{\left\lceil {#1}\right\rceil}

\newcommand{\hr}[0]{\hookrightarrow}

\newcommand{\ord}{\operatorname{ord}}

\newcommand{\tr}[0]{\text{Tr}}
\newcommand{\spn}{\operatorname{Span}}
\newcommand{\pf}[2]{\pa{\frac{#1}{#2}}}

\newcommand{\nequiv}[0]{\not\equiv}

\newcommand{\matt}[4]{
\left(
\begin{matrix}
{#1}&{#2}\\
{#3}&{#4}
\end{matrix}
\right)}

\newcommand{\iy}[0]{\infty}

\newcommand{\fc}[2]{\frac{#1}{#2}}

\newcommand{\hra}[0]{\hookrightarrow}

\newtheorem{theorem}{Theorem}
\newtheorem{lemma}[theorem]{Lemma}
\newtheorem{corollary}[theorem]{Corollary}

\newtheorem{proposition}[theorem]{Proposition}

\renewcommand{\ss}{\hspace*{2pt}}
\theoremstyle{remark}

\newtheorem*{remarkstar}{Remark}
\newtheorem*{remarks}{Remarks}
\newtheorem{example}{Example}
\numberwithin{theorem}{section} \numberwithin{equation}{section}

\newcommand{\spt}{spt}

\newcommand{\olm}{\Omega_\ell}
\newcommand{\olmeven}{\Omega_{\ell}^{\text{even}}}
\newcommand{\olmodd}{\Omega_{\ell}^{\text{odd}}}
\newcommand{\kl}{k_\ell}
\newcommand{\klodd}{k_\ell^{\text{odd}}}

\newcommand{\Ul}{\ss | \ss U(\ell)}

\newcommand{\Dl}{\ss | \ss D_1(\ell)}
\newcommand{\Dlr}{\ss | \ss D_r(\ell)}
\newcommand{\Xl}{\ss | \ss X_r(\ell)}

\newcommand{\yspt}{\ss | \ss Y_1(\ell)}
\newcommand{\Yl}{\ss | \ss Y_r(\ell)}
\newcommand{\Ylr}{\ss | \ss Y_r(\ell)}
\newcommand{\pl}{P_{\ell}}
\newcommand{\lb}{L_{\ell}}

\newcommand{\klm}{k_{\ell}(\spt,m)}
\newcommand{\kld}{k^{\dagger}_{\ell}}

\renewcommand{\and}{\text{ and }}
\newcommand{\zlm}{\mathbb{Z}/\ell^m \mathbb{Z}}
\newcommand{\laodd}{\Lambda_{\ell}^{\text{ odd}}}
\newcommand{\laeven}{\Lambda_{\ell}^{\text{ even}}}

\newcommand{\dl}{d_\ell}
\newcommand{\ml}{M_{\ell+1} \cap \mathbb{Z}_{(\ell)} [[q]]}
\newcommand{\zl}{\mathbb{Z}_{(\ell)} [[q]]}
\newcommand{\Ll}{L_\ell}

\newcommand{\bl}{b_{\ell}}

\newcommand{\st}{\ss : \ss}
\newcommand{\pieven}{\Pi_\ell^{\text{even}}}
\newcommand{\piodd}{\Pi_\ell^{\text{odd}}}

\newcommand{\sspt}{\mathcal{S}(\spt)}

\newcommand{\cals}{\mathcal{S}}
\newcommand{\dspt}{d_{\ell}(\spt)}
\newcommand{\dr}{d_{\ell}(r)}
\newcommand{\drp}{d_{\ell}'(r)}

\newcommand{\ordl}{\ord_{\ell}}
\newcommand{\ff}[2]{\left\lfloor\frac{#1}{#2}\right\rfloor}
\newcommand{\seven}{\mathcal{S}^{\text{even}}}
\newcommand{\sodd}{\mathcal{S}^{\text{odd}}}

\newcommand{\contains}{\supseteq}

\newcommand{\ink}[1]{\in^{#1}}

\makeatletter
\def\imod#1{\allowbreak\mkern10mu({\operator@font mod}\,\,#1)}
\makeatother

\newenvironment{newlist}{\begin{list}{\hspace{5pt}(\arabic{itemcounter})} {\usecounter{itemcounter}\itemsep=0.1in\leftmargin=-0.0em}} {\end{list}}
\newenvironment{newlist2}{
 \begin{list}
{\alph{itemcounter2}.} {\usecounter{itemcounter2}\itemsep=6pt}}{\end{list}}

\newcommand{\llabel}[1]{\label{#1}}

\newcommand{\Keywords}[1]{\par\noindent \newline
{\small{\em Keywords\/}: #1}}

\begin{document}

\newcounter{itemcounter}
\newcounter{itemcounter2}

\title[$\ell$-adic properties of partition functions]
{$\ell$-adic properties of partition functions}
    \author{Eva Belmont, Holden Lee, Alexandra Musat, Sarah Trebat-Leder}

\address{Department of Mathematics, Massachusetts Institute of Technology, Cambridge, MA 02139}
\email{ebelmont@mit.edu}

\address{Trinity College, Cambridge CB2 1TQ, UK}
\email{holdenlee@alum.mit.edu}

\address{Department of Mathematics, Stanford University, Stanford, CA  94305}
\email{amusat@stanford.edu}

 \address{Department of Mathematics, Emory University, Emory, Atlanta, GA 30322}
\email{strebat@emory.edu}


\begin{abstract}
Folsom, Kent, and Ono used the theory of modular forms modulo $\ell$ to establish remarkable ``self-similarity'' properties of the partition function and give an overarching explanation of many partition congruences. We generalize their work to analyze powers $p_r$ of the partition function as well as Andrews's $\spt$-function. 
By showing that certain generating functions reside in a small space made up of reductions of modular forms, we set up a general framework for congruences for $p_r$ and $\spt$ on arithmetic progressions of the form $\ell^mn+\delta$ modulo powers of $\ell$.
Our work gives a conceptual explanation of the exceptional congruences of $p_r$ observed by Boylan, as well as striking congruences of $spt$ modulo 5, 7, and 13 recently discovered by Andrews and Garvan.

\end{abstract}

\maketitle

\noindent
\section{Introduction}
A \emph{partition} of a nonnegative integer $n$ is a non-increasing sequence of positive integers summing to $n$. Letting $p(n)$ denote the number of partitions of $n$, we have that the generating function for $p$ is
\begin{equation}
\prod_{n=1}^{\iy} \rc{1-q^n}=\sum_{n=0}^{\iy}p(n)q^n.
\end{equation}
Ramanujan congruences such as $$p\pa{5^mn+\de_5(m)}\equiv 0\pmod{5^m},$$ where $24\de_5(m)\equiv 1\pmod{5^m},$ have been known for quite some time. Using the theory of $\ell$-adic modular forms, Folsom, Kent, and Ono established a general framework for partition congruences in~\cite{ono1}, which not only explains this phenomenon but also allows them to show additional congruences such as 
\[
p(13^4n+27371)\equiv 45p(13^2n+162)\pmod{13^2}.
\]
More precisely, they defined a special sequence of power series $\lb(b,z)$ related to $p(n)$, whose $\ell$-adic limit resides in a finite-dimensional space. In particular, when the dimension of this space is zero, which happens when $\ell = 5,7,11$, one obtains the famous congruences of Ramanujan modulo powers of $\ell$. Note that the congruences obtained in \cite{ono1} are systematic ones which can be controlled by the use of operators. They are not to be confused with congruences of the type
\begin{equation*} p(59^4 \cdot 13n+111247) \equiv 0 \pmod {13},
\end{equation*} proven by Ono in \cite{onoannals}, which are explained by a different phenomenon.

We generalize their techniques, utilizing careful refinements due to Boylan and Webb in~\cite{boylanwebb}, to establish a framework for the related functions $p_r$ and $\spt$, to be defined below.

\subsection{Powers of the partition function}
For positive values of $r$, we let $p_r(n)$ denote the coefficients of the $r^{th}$ power of the partition generating function:
\begin{equation} 
\prod_{n=1}^{\iy} \rc{(1-q^n)^r}=\sum_{n=0}^{\iy} p_r(n) q^n.
 \end{equation}
We can think of $p_r(n)$ as the number of $r$-colored partitions of $n$.  
Congruences for $p_r(n)$ have been well-studied; for example, in~\cite{atkin2}, Atkin showed congruences for $p_r(n)$ modulo powers of 5, 7, 11, and 13, while in~\cite{boylan}, Boylan classified Ramanujan-type congruences for $p_r$ under mild assumptions. 

To state our results, we define a sequence of generating functions related to $p_r(n)$ which are analogous to those that appear in~\cite{ono1} and~\cite{boylanwebb},
\begin{equation}
\pl(r,b;z):=\sum_{n=0}^{\iy} p_r\pf{\ell^bn+r}{24}q^{\frac{n}{24}},
\end{equation}
where $q$ denotes $e^{2\pi i z}$. Our main theorem is that these functions, reduced modulo powers of $\ell$, reside in a relatively small space. 

Fixing integers $b \geq 0, m \geq 1$, following \cite{boylanwebb} we define 
\begin{align}
\laodd(r, b, m)&:=\spn_{\zlm}\{\eta(\ell z)^r \pl(r,\beta;z) \pmod{\ell^m}\st \beta \geq b \text{ odd}\}\\ 
\laeven(r, b, m)&:=\spn_{\zlm}\{\eta(z)^r \pl(r,\beta;z) \pmod{\ell^m}\st \beta \geq b \text{ even}\}.
\end{align}
Here, $\eta(z) := q^{\frac{1}{24}}\prod(1-q^n)$ is Dedekind's eta-function. In Section \ref{S4} we will define $\dl(r)$ and $\dl'(r)$, which are related to the dimension of the kernel of an important operator.

The following is our main result concerning the partition functions $p_r(n)$.

\begin{theorem}\llabel{pr1}
Let $\ell \geq r + 5$ be prime and let $m, r \geq 1$.  Then there is an integer $\bl'(r, m)$ that satisfies the following.
\begin{newlist}
\item The nested sequence of $\Z /\ell^m\Z$--modules 
\[
\laodd(r, 1, m) \supeq \laodd(r, 3, m) \supeq \ldots
\supeq\laodd(r, 2b+1,m)\supeq \cdots
\]
is constant for all $b$ with $2b + 1 \geq \bl'(r, m)$.  Moreover, if one denotes the stabilized $\zlm$--module by $\olmodd(r, m)$, then its rank, $r_\ell(r)$, is at most 
\[
R_{\ell}(r):=\begin{cases}
\ff{\klodd(r,1)}{12}-\ff{r(\ell^2-1)}{24\ell} &\klodd(r,1)\nequiv 2\pmod{12},\bigskip\\
\ff{\klodd(r,1)}{12}-1-\ff{r(\ell^2-1)}{24\ell} &\klodd(r,1)\equiv 2\pmod{12},\end{cases}\]
where
\[\klodd(r, 1) := \begin{cases} (\frac{r+1}{2})(\ell-1) &\text{ if $r$ is odd,} \\ (\frac{r+2}{2})(\ell-1) & \text{ if $r$ is even.}\end{cases}
\]
\item The nested sequence of even $\zlm$--modules $\{\laeven(r, b, m)\st b \geq \bl'(r, m)\}$ is constant for all $b$ with $2b \geq \bl'(r, m)$.  If we denote the stable module by $\olmeven(r, m)$, then its rank is at most $R_\ell(r)$.  
\end{newlist}
Moreover, if we define $\bl(r, m)$ to be the least such integer, then 
\begin{align*}
\bl(r,1)&\le 2\dr+1\\
\bl(r,m)&\le 2(\dr+1)+2(d_{\ell}'(r)+1)(m-1).
\end{align*}
\end{theorem}

\begin{remarks} \hspace{0pt}
\begin{newlist}
\item From the definition of $d_{\ell}(r)$ and basic properties of linear maps, we get the trivial bound
\[
d_{\ell}(r)\le\dim(S_{\ell-1}).
\]
The constant $d_{\ell}(r)$ can always be calculated algorithmically. For example, when $\ell \leq 29$ and $r$ is such that $r_\ell(r) = 1$ we compute that $d_\ell(r) = 0$. In most of these cases $d_{\ell}'(r)=0$ as well; in general $d_{\ell}'(r)\le d_{\ell}(r)+1$. 
The question of resolving when $d_\ell(r)$ is zero is an open problem.
\item We shall make use of the theory of modular forms modulo $\ell$, which is coherent and well developed for $\ell \geq 5$.
The restriction that $\ell\ge r+5$, however, is necessary for the calculations in some of our proofs. We make no claims for primes $\ell<r+5$.
\end{newlist}
\end{remarks}

When the dimension $r_\ell(r)$ is zero or one, we get the following congruences.  

\begin{corollary}\llabel{pr-corollary}
Fix $r, m \geq 1$, and let $\ell$ be a prime such that $r_\ell(r) \leq 1$.
If $b_1, b_2 \geq \bl(r,m)$ and $b_1 \equiv b_2 \pmod{2}$, then there is an integer $C_\ell(r, b_1, b_2, m)$ such that for all $n$, we have
$$p_r\pf{\ell^{b_1}n + r}{24} \equiv  C_\ell(r,b_1, b_2, m)\cdot p_r\pf{\ell^{b_2}n + r}{24} \pmod{\ell^m}.$$
If $r_\ell(r) = 0$, then $C_\ell(r, b_1, b_2, m)= 0$. 
\end{corollary}
\begin{remarks} \hspace{0pt}
\begin{newlist}
\item
When $r_\ell(r) \leq 1$, Corollary 1.2 gives rise to natural orbits. 
In fact, there exist $C_\ell^\text{even}(r, m)$ and  $C_\ell^\text{odd}(r, m)$
such that if $b_1, b_2$ are even, we have that $C_\ell(r, b_1, b_2, m) =
C_\ell^\text{even}(r, m)^\frac{b_2 - b_1}{2}$ and if $b_1, b_2$ are odd, we have
that $C_\ell(r, b_1, b_2, m) = C_\ell^\text{odd}(r, m)^\frac{b_2 - b_1}{2}$. 
Hence there are at most two orbits, possibly depending on the parity of $b$. 
In fact, there is only one orbit: if $b$ is odd then Lemma \ref{2.1}, together with
the notation of \eqref{l-recursive}, shows that
\begin{align*}
L_\ell(r,b+1,z)\equiv L_\ell(r,b,z)\Ul &\equiv C_\ell^{odd}(r,b,b+2,m)
L_\ell(r,b+2,z)\Ul\\
& \equiv C_\ell^{odd}(r,b,b+2,m)L_\ell(r,b+3,z),\end{align*}
so $C_{\ell}^{\text{odd}}=C_{\ell}^{\text{even}}$.

We note that theorems of this type, and more generally all of the theorems in this paper, can be implemented to get estimates of when the coefficients of the partition generating functions are in any of the residue classes modulo $\ell^m$, not just zero.
\item
Following \cite{olsson} and \cite{boylan}, a pair $(r,\ell)$ is called \emph{exceptional} if there is a congruence of the form $p_r(\ell n+a)\equiv 0\pmod{\ell}$, and \emph{superexceptional} if the congruence is not explained by any of the three criteria Boylan gives in Theorem 2.3 in~\cite{boylan}.  We find that for all pairs $(r, \ell)$ which are exceptional but not superexceptional and satisfy $\ell \geq r + 5$, we have $r_\ell(r) =R_{\ell}(r)= 0$.   For the first two superexceptional pairs, $(5, 23)$ and $(7, 19)$, we find that $r_\ell(r) = 1$. In other words, for $m=1$ we have $C_\ell(r,m)=0$, but for $m>1$, $C_\ell(r,m)\ne 0$. 
%
\end{newlist}
\end{remarks}

Finally, as a direct corollary of Theorem 7.1 in~\cite{ono1}, in the case $r_\ell(r) = 1$ we have that the forms $\pl(r,b;24z)\pmod{\ell^m}$ converge to Hecke eigenforms as $b,m\to \iy$. This is the analogue of Theorem~1.3 in~\cite{ono1}.
For even $r$, $\pl(r, b, 24z)$ has integral weight, while for odd $r$, it has half-integral weight, so we first review the definitions of Hecke operators.  
For odd $r$ and $c$ a prime not dividing the level, recall that for $\la\in \Z$ and $c$ prime, the Hecke operator $T(c^2)$ of weight $\la+\rc 2$ with Nebentypus $\chi$ is defined by
\begin{equation} 
\label{hecke-operator}
\pa{\sum_n a(n)q^n}\,|\,T(c^2):=\sum_n\pa{
a(c^2n) +c^{\la-1} \pf{(-1)^{\la}n}c\chi(c) a(n) +c^{2\la-1} a(n/c^2)
}q^n.
\end{equation}
For even $r$ and $c$ a prime not dividing the level, recall that the Hecke operator $T(c)$ on the space $M_k^!(\Gamma_0(\ell),\chi)$ is defined as
\begin{equation} 
\label{hecke-operator-2}
\sum a(n)q^n | T(c) := \sum\pa{a(nc) + p^{k-1}\chi(c)a\pa{\frac{n}{c}}}q^n. \end{equation} 
\begin{theorem}\llabel{hecke} \label{1.3}
If $r_\ell(r) = 1$, then $P_{\ell}(r,b;24z)\pmod{\ell^m}$ is an eigenform of all the weight $\kl(r, m)- \frac{r}{2}$ Hecke operators on $\Ga_0(576)$, for $b \geq b_\ell(r,m)$.
\end{theorem}

As an immediate corollary, we get the following congruences for $p_r(n)$.
\begin{corollary}\llabel{hecke-cor-pr}
 Suppose $\ell, r$ are such that $r_{\ell}(r) = 1$ and $m\ge 1$.
 \begin{newlist}
 \item If $r$ is odd and $b\ge \bl(r,m)$, then there is an integer $\la_{\ell}(m,c)$ such that for all $n$ coprime to $c$ we have
\begin{align*}
&p_r\pf{\ell^{b}nc^3+r}{24} \equiv
\la_{\ell}(m,c)
\ba{
p_r\pf{\ell^bnc+r}{24}} \pmod{\ell^m}.
\end{align*}
\item If $r$ is even and $b\ge \bl(r,m)$, then there is an integer $\la_{\ell}(m,c)$ such that for all $n$ coprime to $c$ we have
$$p_r\pf{\ell^b c^2 n + r}{24} \equiv \la_{\ell}(m, c) \ba{p_r\pf{\ell^b c n + r}{24}} \pmod{\ell^m}.$$
\end{newlist}
\end{corollary}

\subsection{The $\spt$ function}
In the previous subsection, we obtained theorems analogous to those proven in~\cite{ono1}, where $p(n)$ is replaced by $p_r(n)$. These theorems rely on the fact that the generating functions for $p_r(n)$ are essentially modular forms. However, it turns out that the general strategy for studying $p_r(n)$ also applies to partition functions which do not directly relate to modular forms, for example, the $\spt$-function introduced by Andrews in~\cite{andrews}.

The $\spt$ function counts the number of smallest parts among the partitions of $n$.  For example, for $n = 3$, we have
\[\underline{3}, \ss \ss \ss 2 + \underline{1}, \ss \ss \ss  \underline{1} + \underline{1} + \underline{1}.\]
The smallest parts are underlined, giving us $spt(3) = 5$. For convenience of notation, we define $s(n):= spt(n)$. Andrews \cite{andrews} proved the following Ramanujan-type congruences
\begin{align*}
s(5n + 4) &\equiv 0 \pmod{5}, \\
s(7n + 5) &\equiv 0 \pmod{7}, \\
s(13n + 6) &\equiv 0 \pmod{13}.
\end{align*}
Garvan recently proved that these three congruences are the simplest examples of elegant systematic congruences modulo arbitrary powers of 5, 7, and 13, namely
\begin{align}
\nonumber
s(5^bn + \delta_5(b)) &\equiv 0 \pmod{5^{\fl{\frac{b+1}{2}}}}, \\
\llabel{garvancongruences}
s(7^bn + \delta_7(b)) &\equiv 0 \pmod{7^{\fl{\frac{b+1}{2}}}},  \\
\nonumber s(13^bn + \delta_{13}(b)) &\equiv 0 \pmod{13^{\fl{\frac{b+1}{2}}}},
\end{align}
where $\delta_{\ell}(b)$ denotes the least nonnegative residue of $24^{-1}$ modulo $\ell^b$.
For all primes $\ell \geq 5$, Ono \cite{ono3} found a systematic modified type of congruence which  Ahlgren, Bringmann, and Lovejoy \cite{abl} have generalized to all powers of primes $\ell \geq 5$.

We would like to apply similar techniques as in the $p_r$ case, but the generating function for $\spt$,
\begin{equation}
\sum s(n) q^n = \pa{\sum_{n \ge 0} p(n)q^n} \pa{\sum_{n = 1}^\infty \frac{q^n  \prod_{m = 1}^{n-1}(1 - q^m)}{1 - q^n}},   \end{equation}
is not modular. Instead, it is essentially the holomorphic part of a weight $\frac{3}{2}$ harmonic Maass form. More precisely, if we work with the sequence
\begin{equation} 
{\bf a}(n) := 12 s(n) + (24n - 1)p(n),
\end{equation}
and define
\begin{equation} 
\alpha(z) := \sum_{n \geq 0} {\bf a}(n) q^{n - \frac{1}{24}}, \end{equation}
then $\alpha(24z)$ is the holomorphic part of a weight $\frac{3}{2}$ weak Maass form, as shown by Bringmann~\cite{bringmann}.

Although harmonic Maass forms are modular, they have Fourier expansions that are not holomorphic. Fortunately, we can return to the theory of modular forms by annihilating the non-holomorphic part. Garvan~\cite{garvan} accomplished this by Atkin's $U(\ell)$--operator, while Ono~\cite{ono3} used the weight $\frac{3}{2}$ Hecke operator $T(\ell^2)$, and Ono and others have used the theory of twists (for example, see Theorems 10.1 and 10.2 of~\cite{ono-maass}). We follow Garvan, and define 
\begin{equation}
 \alpha_\ell(z):=\sum_{n \geq 0}\left(\textbf{a}\left(\ell n - \frac{1}{24}(\ell^2 - 1)\right) - \chi_{12}(\ell) \, \ell \,\textbf{a} \left(\frac{n}{\ell}\right)\right) q^{n - \frac{\ell}{24}}, \end{equation}
where $\chi_{12}(\bullet) := (\frac{12}{\bullet})$. As shown by Garvan in~\cite{garvan}, for primes $\ell\ge 5$, $\al_{\ell}(z)$ is a weakly holomorphic modular form of weight $\frac{3}{2}$ (see the discussion before (1.20) of \cite{garvan}). 
For the primes 2 and 3, a fairly complete theory has been obtained by Folsom and Ono in \cite{folsom}, with an alternate approach for the prime 2 by Andrew, Garvan, and Liang in \cite{agl}. We would like to apply similar techniques as the $p_r$ case to obtain results for $\ell \geq 5$. 

Following Garvan, our main objects of study are the series
\begin{equation}
\pl(\spt,b;z):=\sum_{n\ge -\ell} \pa{{\bf a}\pf{\ell^bn+1}{24}-\chi_{12}(\ell) \ell {\bf a}\pf{\ell^{b-2}n+1}{24}}q^{\frac n{24}}.
\end{equation}
Note that $\pl(\spt,1;z) = \alpha_\ell(z)$. 

As in the $p_r$ case,  fixing integers $b \geq 0$ and $m \geq 1$, we define 
\begin{align*}
\laodd(\spt, b, m)&:=\spn_{\zlm}\{\eta(\ell z) \pl(\spt,\beta;z) \pmod{\ell^m}\st \beta \geq b, \beta \equiv b \pmod{2}\},\\ 
\laeven(\spt, b, m)&:=\spn_{\zlm}\{\eta(z) \pl(\spt,\beta;z) \pmod{\ell^m}\st \beta \geq b, \beta \equiv b \pmod{2}\}
\end{align*}
for odd and even $b$, respectively. The quantity $d_\ell(spt)$ is analogous to $d_\ell(r)$ defined earlier, and is related to the kernel of an important operator we define in Section~\ref{S4}.

\begin{theorem}\llabel{spt1}
If $\ell \geq 5$ is prime and $m\ge 1$, then there is an integer $\bl'(\spt, m)$ such that the following are true.  
\begin{newlist}
\item The nested sequence of $\Z/\ell^m\Z$--modules 
$$\laodd(\spt, 1, m) \supeq \laodd(\spt, 3, m) \supeq \ldots\supeq \laodd(\spt, 2b+1,m)\supeq \cdots $$
is constant for all $b$ with $2b + 1 \geq \bl'(\spt, m)$.  Moreover, if one denotes the stabilized $\zlm$-module by $\olmodd(\spt, m)$, then its rank, $r_\ell(\spt)$, is at most 
\[
R_\ell(\spt):=\begin{cases}
\fl{\frac{\ell + 1}{12}} - \ff{\ell^2 - 1}{24 \ell}& \text{if }\ell\nequiv 1\pmod{12},\bigskip\\
\fl{\frac{\ell + 1}{12}} - 1 - \fl{\frac{\ell^2 - 1}{24 \ell}}& \text{if }\ell\equiv 1\pmod{12}.\end{cases}
\]
\item The nested sequence of even $\zlm$-modules $\{\laeven(\spt, b, m)\st b \geq \bl'(r, m)\}$ is constant for all $b$ with $2b \geq \bl(\spt, m)$.  If we denote the stabilized module by $\olmeven(\spt, m)$, then its rank is at most $R_\ell(\spt)$. 
\end{newlist}
Moreover, if we define $\bl(\spt, m)$ to be the least such integer, then 
$$\bl(\spt, m) \leq  2(d_\ell(spt) + 1)m+1.$$
\end{theorem}
\begin{remarkstar} 
From the definition of $d_{\ell}(spt)$, we get the trivial bound 
\[d_{\ell}(spt)\le \dim(S_{\ell+1}).\]
\end{remarkstar}

For $\ell = 5, 7$, or $13$, we have that $r_\ell(\spt) = 0$ and $d_{\ell}(\spt)=0$, and therefore our generalization of the theory of \cite{ono1} reproduces Garvan's congruences 
but with the power of the modulus is one smaller. It is interesting to note that $\ell$-adic theory can replace the theory of modular equations.
With slightly more extra input, we would be able to reproduce the full congruences.

\begin{corollary}
\llabel{dim0}
For the primes $\ell = 5, 7$, and $13$, we have 
\[s(\ell^b n+\de_{\ell}(b)) \equiv 0 \pmod{\ell^{\ff{b-1}2}}.\]
\end{corollary}

For $\ell = 11, 17, 19, 29, 31$, or $37$, we have that $r_\ell(\spt) = 1$, and therefore we get the following theorem. 
\begin{corollary}
\llabel{1.1}
Let $\ell$ be one of the primes above and $m \geq 1$.
If $b_1, b_2 \geq \bl(\spt,m)$ and $b_1 \equiv b_2 \pmod{2}$, then there is an integer $C_\ell(m, b_1, b_2)$ such that for all $n$, we have
 \begin{multline*}
 s\pa{\ell^{b_1} n + \de_{\ell}(b_1)}- \chi_{12}(\ell) \, \ell\, s\pa{\ell^{b_1-2}n + \de_{\ell}(b_1-2)} \equiv \\
  C_\ell(m, b_1, b_2)\cdot \ba{s\pa{\ell^{b_2} n + \de_{\ell}(b_2)}-\chi_{12}(\ell)\, \ell\, s\pa{\ell^{b_2-2}n +\de_{\ell}(b_2-2)}} \pmod {{\ell}^m}. 
\llabel{dim1eq}
\end{multline*}
 \end{corollary}
\begin{remarks} \hspace{0pt}
\begin{newlist}
 \item
Also note that when $m=1$, the above formula dramatically simplifies as
$$ s(\ell^{b_1}n + \de_{\ell}(b_1)) \equiv C_\ell(1,b_1,b_2) \cdot s(\ell^{b_2}n + \de_{\ell}(b_2)) \pmod{\ell}.$$
\item Just as in the case of $p_r$, when $r_\ell(spt) \leq 1$ then the theorem above gives rise to natural orbits.
\end{newlist}
\end{remarks}

Finally, we show that for small $\ell$, the forms $\pl(\spt,b;z)\pmod{\ell^m}$ converge to Hecke eigenforms as $b,m\to \iy$, where the Hecke operators are as defined in \eqref{hecke-operator}. This is analogous to Theorem 1.3 in \cite{ono1}.
\begin{theorem}\llabel{hecke-spt}
If $5\le \ell\le 37$ is prime with $\ell\ne 23$, then $P_{\ell}(\spt,b;24z)\pmod{\ell^m}$ for $b \geq b_\ell(m)$ is an eigenform of all the weight $\kl(\spt, m)-\rc 2$ Hecke operators on $\Ga_0(576)$.
\end{theorem}
As an immediate corollary, we get the following congruences for $s(n)$.
\begin{corollary}\llabel{hecke-cor-spt}
Suppose $\ell = 11, \,17,\,19,\,29,\,31,$ or $37$, and $m\ge 1$. If $b\ge \bl(\spt,m)$, then there is an integer $\la_{\ell}(m,c)$ such that for all $n$ coprime to $c$ we have
\begin{align*}
&s\pf{\ell^{b}nc^3+1}{24}-\chi_{12}(\ell)\ell s\pf{\ell^{b-2}nc^3+1}{24} \\
&\equiv
\la_{\ell}(m,c)
\ba{
s\pf{\ell^{b}nc+1}{24}-\chi_{12}(\ell) \ell s\pf{\ell^{b-2}nc+1}{24}
} \pmod{\ell^m}.
\end{align*}
\end{corollary}
\begin{remarkstar}
In the case when $m = 1$, this simplifies to 
\[
s\pf{\ell^{b}nc^3+1}{24} \equiv
\la_{\ell}(m,c)
s\pf{\ell^{b}nc+1}{24}
\pmod{\ell}.
\]
\end{remarkstar}

\subsection{Examples}
Here, we give numerical examples of the main theorems in this paper. In all cases, using the methods in \S 6 of \cite{boylanwebb}, we find that $d_{\ell}(r)$ and $d_\ell(spt)$ are zero.
\begin{example} We illustrate Corollary \ref{pr-corollary} in the case that $\ell = 13$, $r = 2$, and $m=1$. We calculate that
$$P_{13}(2, 4; z) \equiv 10 \cdot P_{13}(2, 2, z) \equiv 1 + 4q + q^2 + \cdots \pmod{13}$$ 
and therefore
$$ p_2(13^4n + 26181)\equiv 10 \cdot p_2(13^2n + 155) \pmod{13}.$$
More generally, for every even $b_1, b_2 \geq 1$, we have
$$ p_2\pa{\frac{13^{b_1}n+2}{24}} \equiv 10^{\frac{b_1 - b_2}{2}} p_2\pa{\frac{13^{b_2}n+2}{24}} \pmod{13}.$$
 \end{example}
\begin{example} 
We illustrate Corollary \ref{1.1} in the case that $m=1$ and $\ell=11$. We calculate that
$$P_{11}(\spt, 2;z) \equiv P_{11}(\spt, 4;z) \equiv 4q + 7q^2 + 7q^3 + \cdots \pmod{11}$$
and therefore
$$s(11^2n+116) \equiv s(11^4n+14031) \pmod {11}.$$
More generally, for every even $b_1, b_2 \geq 2$, we have 
$$s\pa{\frac{11^{b_1} n + 1}{24}} \equiv s\pa{\frac{11^{b_2} n + 1}{24}} \pmod {11}.$$
\end{example}
\begin{example} 
We illustrate Theorem \ref{hecke-spt} for $\ell = 17$ when $b = 2$ and $c = 5$.  We calculate that 
$$P_{17}(\spt,2; 24z) \ss | \ss T(5^2) \equiv 2 P_{17}(\spt, 2; 24z) \equiv 13q^{23} + 13q^{71} + 4q^{119} + 8q^{143} +  \cdots \pmod{17}$$
and therefore
\[s(28599 + 36125 n) \equiv 2 \cdot s(1144+1445n) \pmod{17}\]
for $n$ relatively prime to 5. 
\end{example}
\subsection{Outline}
In Section~\ref{S2} we introduce the sequence of functions $\lb(r,b;z)$ and $\lb(\spt, b;z)$, which are related to the functions $\pl(r,b;z)$ and $\pl(\spt,b;z)$, and which are obtained by repeatedly applying certain operators $U(\ell)$ and $D(\ell)$. We also recall various facts about filtrations of modular forms that we will need. 
In Section~\ref{S3}, we show that $U(\ell)$ and $D(\ell)$ preserve spaces of modular forms, and in Section~\ref{S4}, we show that iterating these operators results in spaces $\olm(r,m)$ and $\olm(\spt,m)$ with small rank. Since $\lb(\spt,b;z)$ and $\lb(r,b;z)$ reside in these spaces, this proves the finiteness portions of Theorems~\ref{pr1} and~\ref{1.1} and their corollaries. In Section~\ref{S5}, 
we prove the bounds on $\bl(r,m)$ and $\bl(\spt, m)$ given in these two theorems. In Section~\ref{S6}, we give proofs of the main theorems.

\section*{Acknowledgements}
The authors would like to thank Ken Ono for hosting the Emory REU, where the research was conducted, and for his guidance and comments on this paper. We would also like to thank Zachary Kent for useful discussions, M. Boylan and J. Webb for providing a preprint of their paper which was useful for our research, and the National Science Foundation for funding the REU. Finally, we would like to thank the referee for 
helping to improve the bound for $b_{\ell}(r,m)$ in Theorem~\ref{pr1}.

\section{Combinatorial and Modular Properties of Important Generating
Functions}\llabel{S2}
\subsection{Basic definitions}
The proofs of the theorems in this paper rely on a detailed study of a
peculiar sequence of power series. To define them, we need combinatorial
properties of some very special generating functions. 

Throughout the paper, we will consider $\ell \ge 5$ prime.  In the case of $p_r$, we will always be considering $\ell \geq r + 5$.  Recall that the Atkin $U(\ell)$-operator acts on $q$-series by
\begin{equation}
\pa{ \sum a(n) q^n} \Ul := \sum a(n\ell) q^n.
\end{equation}

Following \cite{ono1}, we define
\begin{equation} \Phi_{\ell}(z):=
\frac{\eta(\ell^2z)}{\eta(z)}, \end{equation}
and define the operator $D_r(\ell)$ by
\begin{equation} f \Dlr := (f \cdot (\Phi_\ell(z))^r) \Ul. \end{equation}
To study $p_r$ we will consider the sequence of operators $U(\ell),
D_r(\ell), U(\ell), D_r(\ell), \ldots$.  To study $spt$ we will consider
the same sequence, with $r = 1$. 
It will occasionally be useful for us to think of this sequence as
simply repeated use of the operators $U(\ell) \circ D_r(\ell)$ or
$D_r(\ell) \circ U(\ell)$. To this
end, we define
\begin{align}
f \Xl &:= (f \Ul ) \Dlr \text{ and}
\\f \Yl &:= (f \Dlr ) \Ul.
\end{align}

We now define for $b \ge 1$ two special sequences of functions. 
Let $L_\ell(spt,1;z) :=  \eta(\ell z)\alpha_\ell(z)$, and for $b \ge 2$,
define $L_\ell(spt,b;z)$ by
 \begin{equation}
 L_\ell(spt,b;z) := \begin{cases} L_\ell(spt,b-1;z) \Ul & \text{ if $b$
 is even,} \\ 
 L_\ell(spt,b-1;z) \Dl & \text{ if $b$ is odd.} 
 \end{cases}
 \end{equation}
Let $L_\ell(r,0;z) := 1$, and for
$b \ge 1$, define $L_\ell(r,b;z)$ by
 \begin{equation}\label{l-recursive}
 L_\ell(r,b;z) := \begin{cases} L_\ell(r,b-1;z) \Ul & \text{ if $b$ is
 even,} \\ 
 L_\ell(r,b-1;z) \Dlr & \text{ if $b$ is odd.} 
 \end{cases}
 \end{equation}
The following lemma will relate $L_\ell(spt,b;z)$ and $L_\ell(r,b;z)$ to
our main objects of interest, the functions 
$P_\ell(spt,b;z)$ and $P_\ell(r,b;z)$.

\begin{lemma}
\llabel{2.1}
For $b \ge 1$, we have that
\begin{equation}
L_\ell(\spt,b;z) = \begin{cases}\eta(z) \cdot P_\ell(\spt,b;z) & \text{ if
$b$ is even,} \\
  \eta(\ell z) \cdot P_\ell(\spt,b;z) & \text{ if $b$ is odd.} 
 \end{cases}
 \end{equation}
For $b\ge 0$, we have that
\begin{equation}
L_\ell(r,b;z) = \begin{cases}\eta^r(z) \cdot P_\ell(r,b;z) & \text{ if
$b$ is even,} \\
  \eta^r(\ell z) \cdot P_\ell(r,b;z) & \text{ if $b$ is odd.} 
 \end{cases}
 \end{equation}
\end{lemma} 
\begin{proof}
Note that by definition, we have
\begin{align*}
L_\ell(\spt,1;z)&= \eta(\ell z) \cdot P_\ell(\spt,1;z),\\
L_{\ell}(r,0;z)&= \eta^r(z)\cdot P_{\ell}(r,0;z).
\end{align*}
To prove the lemma in general, we use induction on $b$ and the following
fact about the $U(\ell)$-operator.  If $F(q)$ and $G(q)$ are formal power
series with integer exponents, then
\begin{equation}
(F(q^\ell) \cdot G(q)) \Ul = F(q) \cdot (G(q) \Ul).
\end{equation} 
The lemma now follows by direct computation.  \end{proof}

We now prove a result about the modularity of the functions $L_\ell(\spt,b;z)$ and
$L_\ell(r,b;z)$. Using standard notation, we
denote by $M_k(\Gamma_0(N))$ the space of holomorphic modular
forms of weight $k$ on $\Gamma_0(N)$, and we denote by $M_k^{!}(\Gamma_0(N))$ the
space of weakly holomorphic modular forms on $\Gamma_0(N)$  of weight $k$, i.e. those forms
whose poles (if any) are supported at the cusps of $\Gamma_0(\ell)$.
%
\begin{lemma}

\llabel{2.2}
If $b\ge 1$ is a positive integer, then $L_\ell(\spt,b;z)$ is
in $M_2^{!}(\Gamma_0(\ell)) \cap \mathbb{Z} [[q]]$  and $L_\ell(r,b;z)$ is in
$M_0^{!}(\Gamma_0(\ell)) \cap \mathbb{Z} [[q]]$.
\end{lemma}
\begin{proof}

For $b=1$, we have that $L_\ell(spt,b;z)=\alpha_\ell(z) \eta(\ell z)$,
which is in $M_2^!(\Ga_0(\ell))$ by a result from~\cite{garvan}.
By well-known facts about the Dedekind $\eta$--function (for example,
see Theorem 1.64 and 1.65 in \cite{ono2}), it also follows
that $\Phi_\ell(z) \in M_0^{!}(\Gamma_0(\ell^2))$. Moreover, it is
well-known (see for example Lemma 2.1 in \cite{boylanwebb}) that
\begin{equation*}
U(\ell): M_k^{!}(\Gamma_0(\ell^2)) \to M_k^{!}(\Gamma_0(\ell)),
\end{equation*} and that
\begin{equation*}
U(\ell): M_k^{!}(\Gamma_0(\ell)) \to M_k^{!}(\Gamma_0(\ell)).
\end{equation*}
Combining the two facts above for $k = 2$, an inductive argument shows
that $L_\ell(spt,b;z)$ is in $M_2^{!}(\Gamma_0(\ell)) \cap \mathbb{Z}
[[q]]$ for all $b \geq 1$. The case $b=0$ is obvious for
$L_\ell(r,b;z)$, so the analogous result follows.
\end{proof}

\subsection{Filtrations}
The theory of filtrations has classically been used to understand
modular forms modulo $\ell$. Following \cite{serre} and \cite{sd}, we
review some of the most pertinent facts.


For $f \in M_{k} \cap \mathbb{Z}[[q]]$, define the
filtration of $f$ modulo $\ell$ by 
\begin{equation}
\omega_\ell(f) := \inf_{k \geq 0}\{k : f \equiv g \imod{\ell} \text{ for
some } g \in M_{k} \cap \mathbb{Z}[[q]]\}.
\end{equation}
If $f\equiv 0\imod{\ell}$ then define $\omega_{\ell}(f)=-\iy$. 
Note that if $f \equiv g \pmod{\ell}$ and $g \in M_{k}$, then we have 
$\omega_{\ell}(f) \equiv k \pmod{\ell - 1}$.  

We use filtrations to understand how $U(\ell)$ acts on modular forms
modulo $\ell$. The following is Lemme 2 on p. 213 of
\cite{serre}; it tells us how $U(\ell)$ decreases the filtration.
\begin{lemma}
\llabel{filtration-1}
If $f \in M_k \cap \mathbb{Z}[[q]]$, then we have $\omega_\ell(f \Ul) \leq \ell
+ \frac{\omega_\ell(f) - 1}{\ell}$.
\end{lemma}
Next we will need the following facts from \S 2.2 Lemme 1 of
\cite{serre}, which shows how the $\theta$ operator acts on this graded
ring.
\begin{lemma}\llabel{serre-filtrations} Letting $\theta
f:=q\frac{d}{dq}f$, the following hold.  
\begin{newlist} 
\item We have $\omega_{\ell}(\theta f) \leq \omega_{\ell}(f)
+ \ell + 1$, with
equality if and only if $\omega_{\ell}(f) \not\equiv 0 \pmod \ell$.
\item We have $\omega_{\ell}(f^i) = i \omega_{\ell}(f)$ for all $i
\geq 1$.
\end{newlist}
\end{lemma}
In the case of $\spt$, we will use the $U(\ell)$ and $D_1(\ell)$ operators
to decrease the filtration to $\ell+1$, at which point it is stable. In the case
of $p_r$, applying the operators decreases the filtration to $\ell -1$, at which
point we have
$$ \omega_\ell(r,b;z) = \begin{cases} \ell-1 &\text{ if $b$ is even,} \\ \klodd(r,1) &\text{ if $b$ is odd.} \end{cases} $$
The following lemma describes this process in both cases.
\begin{lemma} \llabel{filtration-2} Let $f\in\Z[[q]]$ be a modular form and $\ell$ be prime.
\begin{newlist}
\item Suppose $\ell \geq 5$ and $\om_{\ell}(f)\equiv 2\pmod{\ell-1}$.
\begin{enumerate}[a.]
\item If $\om_{\ell}(f)=\ell+1$, then we have $\omega_\ell(f \Ul) = \ell +
1$ and $\omega_{\ell}(f\,|\,X_1(\ell)) \le \ell+1$. Hence $U(\ell)$ is a bijection on weight
$\ell + 1$ modular forms modulo $\ell$.  
\item If $\om_{\ell}(f)=\ell+1$, then we have $\omega_\ell(f \Dl) \leq \ell + 1$.
\item If $\omega_\ell(f) > \ell + 1$, then we have $\omega_\ell(f \,|\,X_1(\ell)) < \omega_\ell(f)$.  
\end{enumerate}
\item Suppose $\ell \geq r+5$ and $\omega_\ell(f)\equiv 0\pmod{\ell - 1}$.  
\begin{enumerate}[a.]
\item If $\omega_\ell(f) \leq \klodd(r, 1)$, then we have $\omega_\ell(f \Ul)
\leq \ell - 1$ and $\omega_\ell(f \ss | \ss X_r(\ell))\le \klodd(r,1)$.
\item If $\omega_\ell(f) = \ell - 1$, then we have $\omega_\ell(f \Dlr)
\leq \klodd(r, 1)$.
\item If $\omega_\ell(f) > \klodd(r, 1)$, then we have $\omega_\ell(f \ss | \ss X_r(\ell)) < \omega_\ell(f)$. 
\end{enumerate}
\end{newlist}
\end{lemma}
\begin{proof}
To prove part 1(a), use Lemma \ref{serre-filtrations}(1) repeatedly to see that
$\omega_{\ell}(\theta^i f) = (i+1)(\ell+1)$ for $i \leq \ell - 1$. The identity
$(f \,|\, U)^\ell = f - \theta^{\ell - 1}f$, also from \cite{serre},
gives $\omega_{\ell}((f  \,|\,  U)^\ell) =\max(f,\theta^{\ell-1}f)=
\ell(\ell+1)$. By Lemma~\ref{serre-filtrations}(2), we get $\omega_{\ell}(f  \,|\,  U) = \ell + 1$. 
From this and part 1(b) below, we have that  $\om_{\ell}(f\,|\,X_1(\ell))\le \ell+1$.

To prove part 1(b), first note that 
\begin{equation}
\label{phi}
\Phi_\ell(z) \equiv \Delta(z)^{\frac{\ell^2 - 1}{24}} \pmod{\ell}.
\end{equation}
By Lemma~\ref{filtration-1} we have 
\[\omega_\ell \left(
(\Delta(z)^{\frac{\ell^2  - 1}{24}} \cdot f) \Ul \right)
\leq \ell + \frac{\frac{\ell^2 - 1}2 + \ell}{\ell} \leq \frac{3}{2}\ell + 1
< (\ell + 1) + (\ell - 1)\]
for $\ell > 2$.  Since the filtration is congruent to $\ell+1$ modulo $\ell-1$, the result follows.

For part 1(c), write $\omega_\ell(f) = k(\ell - 1) + 2$ with $k > 1$.  Using Lemma \ref{filtration-1}, we compute that $$\omega_\ell(f \,|\,X_1(\ell)) \leq \ell + 
\frac{\pa{\ell + \frac{k(\ell - 1) +1}{\ell}} + \frac{\ell^2 - 1}{2} - 1}{\ell}
= \frac{1}{\ell^2} \pa{(\ell^2 - 1)\frac{3 \ell + 2}{2}+ k(\ell - 1) + 2}.$$
Since for $\ell > 2$ and $k > 1$, we have that $\frac{3 \ell + 2}{2} < k(\ell - 1) + 2$, this part follows.  

For 2(a), using the definition of $\klodd(r,1)$, we get that 
\[
\om_{\ell}(f\Ul)\le \ell+\frac{\pa{\frac r2+1}(\ell-1)-1}{\ell}=\pa{\ell-1}\pa{1+\frac{r+4}{2\ell}}.
\]
Since $\ell
>r+4$, we get that $\omega_\ell(f \Ul) \leq \ell - 1$.  The second part of (a) follows from this and (b) below.

For 2(b), we use Lemma~\ref{filtration-1} to see that $$\omega_\ell(f \Dlr)
\leq \ell + \frac{(\ell-1)+\frac{r(\ell^2 - 1)}{2} -1}{\ell}=(\ell-1)\pa{\frac r2+1+\frac{r+4}{2\ell}}<\klodd(r,1)+\ell-1$$ 
when $\ell > r + 4$.  

Lastly, for part 2(c), using Lemma \ref{filtration-1}, we get that
\[
\omega_\ell(f\Xl) \leq \ell + \frac{ \pa{\ell+\frac{\omega_\ell(f)-1}{\ell}}+r(\frac{\ell^2-1}{2})-1}{\ell}.
\]
Computation shows that the right-hand side is less than $\om_{\ell}(f)$ if and only if $\om_{\ell}(f)>\pf{r+1}{2}(\ell-1)+\frac{\ell-1}2+\frac{r+4}{2}$. Since $r+5\le\ell$, $\ell-1\mid \om_{\ell}(f)$, and $\om_{\ell}(f)>\klodd(r,1)$, we find that this last condition is satisfied.
\end{proof}

\section{The action of $U(\ell)$ and $D(\ell)$}\label{S3}
\subsection{$U(\ell)$ and $D(\ell)$ preserve modular forms}
We show in this section that the functions $\lb(r,b;z)$ and $\lb(\spt,b;z)$ are reductions modulo $\ell$ of certain cusp forms. For the $\spt$ case, we need to show this for the base case $b=1$. For both cases, we need to show the induction step, that $U(\ell)$ and $D(\ell)$ preserve these spaces of modular forms.
It is important to choose weights so that these operators  preserve spaces of those weights modulo powers of $\ell$. For the $\spt$ case we define
\[
\kl(\spt,m):=\ell^{m-1}(\ell-1)+2;
\]
and in the case of $p_r$ we will use the weights $\ell^{m-1}(\ell-1)$ (without special notation). 
For ease of notation, we will write 
\begin{equation}f \ink{j} S   \end{equation}
 if $f$ is congruent modulo $\ell^j$ to a modular form in a space $S$. 

Following~\cite{ono1}, we define
\begin{equation}
A_{\ell} (z):= \frac{\eta^\ell (z)}{\eta( \ell z)}.
\end{equation}
Using standard facts about the Dedekind
eta-function (see for example Theorem 1.64 and 1.65 of \cite{ono2}), we see that $A_{\ell}(z)$ is a holomorphic modular form of weight
$\frac{\ell-1}2$ on $\Gamma_0 (\ell)$ with Nebentypus $\left(\frac{\bullet}{\ell}\right)$. This function is useful to us because it changes the weight of a form 
while preserving the form modulo $\ell$.
\begin{lemma}
\llabel{4.1}
If $\ell \ge 5$ is prime, then $ A_\ell
(z)^{2\ell^{m-1}} \in M_{\ell^{m-1}(\ell-1)} (\Gamma_0(\ell))$. Moreover, 
\begin{equation}
A_\ell (z)^{2\ell^{m-1}} \equiv 1 \pmod {\ell^{m}}.
\end{equation}
\end{lemma}
\begin{proof}
See the proof of Lemma 4.1 in~\cite{ono1}.
\end{proof}

Recall that the weight $k$ slash operator is defined by
\begin{equation*}
(f  \ss|_k\ss A)(z):=(\det A)^{k/2} (cz+d)^{-k} f(Az),
\end{equation*}
where $A=\matt{a}{b}{c}{d}$. As in~\cite{serre}, we define the trace operator $\tr:M_k^!(\Ga_0(\ell))\to M_k^!$ by
\begin{equation}
\tr(f):=f+\ell^{1-\frac{k}{2}}(f\ss|_k\ss W(\ell))\Ul,\quad W(\ell):=\matt0{-1}{\ell}0.
\end{equation}
Note that $\tr$ takes $M_k(\Ga_0(\ell))$ to $M_k$.
\begin{proposition}\llabel{4.2}
If $m\geq 1$, then $L_{\ell}(\spt,1;z) \ink{m}S_{\klm}$.
\end{proposition}
\begin{proof}
Following \cite{garvan}, we define
\begin{equation}
G_{\ell}(z):=\al_{\ell}(z)\cdot \frac{\eta^{2\ell}(z)}{\eta(\ell z)}.
\end{equation}
Note that $G_{\ell}(z)\in M_{\ell+1}(\Ga_0(\ell))$ by Theorem 2.2 in~\cite{garvan}, hence we have 
\[
L_{\ell}(\spt,1;z)A_{\ell}(z)^{2\ell^{m-1}}=G_{\ell}(z)A_{\ell}(z)^{2(\ell^{m-1}-1)}\in M_{\klm}(\Ga_0(\ell))
\]
and
\begin{multline*}
\tr(L_{\ell}(\spt,1;z)A_{\ell}(z)^{2\ell^{m-1}})=L_{\ell}(\spt,1;z)A_{\ell}(z)^{2\ell^{m-1}}\\
+\ell^{1-\frac{(\ell-1)\ell^{m-1}+2}{2}} \pa{G_{\ell}(z)\,|_{\ell+1}\,W(\ell)\cdot A_{\ell}(z)^{2(\ell^{m-1}-1)}\,|_{(\ell^{m-1}-1)(\ell-1)}\,W(\ell)}\,|\,U(\ell)\in M_{\klm}.
\end{multline*}
The first term is congruent to $L_{\ell}(\spt,1;z)$ modulo $\ell^m$ by Lemma~\ref{4.1}, so it suffices to show the second term has valuation at least $m$.

We use equation (2.19) on page $9$ on \cite{garvan},
\begin{equation}
G_\ell \left(\frac{-1}{\ell z}\right)= \displaystyle - (iz\ell)^{\ell+1}\frac{E(q^\ell)^{2\ell}}{E(q)} \displaystyle \sum_{n=-s_\ell} ^{\infty} \left(\chi (\ell) {\bf a}(n) \left( \left(\frac{1-24n}{\ell}\right)-1 \right)+\ell {\bf a} \left(\frac{n+s_\ell}{\ell^2}\right) \right) q^{n+2s_\ell},
\llabel{g}
\end{equation}
where $E(q) := \prod(1-q^n)$, 
and the transformation law for $\eta$ to calculate
\begin{align*}
&\ell^{1-\frac{(\ell-1)\ell^{m-1}+2}{2}} \pa{G_{\ell}(z)|_{\ell+1}W(\ell)\cdot A_{\ell}(z)^{2(\ell^{m-1}-1)}|_{(\ell^{m-1}-1)(\ell-1)}W(\ell)}\,|\,U(\ell)\\
&=\ell^{1-\frac{(\ell-1)\ell^{m-1}+2}{2}} \ba{\ell^{\frac{\ell+1}{2}}S_1(q)\cdot 
(\ell z)^{-(\ell^{m-1}-1)(\ell-1)} \ell^{\frac{\ell^{m-1}(\ell-1)}2}\frac{\eta\pf{-1}{\ell z}^{2\ell(\ell^{m-1}-1)}}{\eta\pf{-1}{z}^{2(\ell^{m-1}-1)}}}\,|\,U(\ell)\\
&=\ell^{1-\frac{(\ell-1)\ell^{m-1}+2}{2}} \ba{\ell^{\frac{\ell^{m-1}(\ell+1)}2}S_2(q)}\,|\,U(\ell) =\ell^{\ell^{m-1}}S_3(q),
\end{align*}
where $S_1,S_2,S_3$ are power series with algebraic integer coefficients. Since $\ell^{m-1}\ge m$, we get the desired conclusion.
\end{proof}
Now we show $D_1(\ell)$ preserves the spaces $S_{\kl(\spt,m)}$.
\begin{lemma}\llabel{bw-3.1}
Suppose $m\ge 1$ and $\Psi(z)\in \Z[[q]]$.
\begin{newlist}
\item Let $\ell\ge 5$ be prime, and suppose that for all $1\le j\le m$, we have $\Psi(z)\ink{j} M_{\kl(\spt,j)}$. Then for all $1\le j\le m$, we have $\Psi(z)\Dl\ink{j}S_{\kl(\spt,j)}$. 
\item
Let $\ell\ge r+3$ be prime, $r\ge 2$, and suppose that we have $\Psi(z)\ink{j} M_{\ell^{j-1}(\ell-1)} $ for $1\le j\le m$. Then for all $1\le j\le m$, we have $\Psi(z)\Ylr\ink{j} S_{\ell^{j-1}(\ell-1)}$.

Furthermore, defining
\[
k_{\ell}^{\text{odd}}(r,j)=\begin{cases}
\pa{\fl{\fc r2}+1}(\ell-1),&j=1\\
\fl{\fc r2}\ell(\ell-1),&j=2\\
\ell^{j-1}(\ell-1),&j\ge 3,
\end{cases}
\]
we have that $\Psi(z)\Dlr \in^j S_{k_{\ell}^{\text{odd}}(r,j)}$ for $1\le j\le m$.
\end{newlist}
\end{lemma}
Note that in (2), one application of $D_r(\ell)$ can increase the weight because $\De^{r\pf{\ell^2-1}{24}}\equiv \Phi_{\ell}^r\pmod{\ell}$ has large weight. However, if we use the operator $Y_r(\ell)$ rather than $D_r(\ell)$, the weight is preserved.
\begin{proof}
The proof is similar to that of Lemma 3.1 in~\cite{boylanwebb}, so we will omit some details and instead note the differences. For (1), we let $g_j$ denote the forms in $M_{\kl(\spt, j)}$ congruent to $\Psi$ modulo $\ell^j$, and for (2), we let $g_j$ denote the forms in $M_{\ell^{j-1}(\ell-1)}$ congruent to $\Psi$ modulo $\ell$ and modulo $\ell^j$, respectively.
\begin{newlist}
\item For the first part, the base case $j=1$ follows by Lemma~\ref{filtration-2}(1)(b).
We write $\Psi(z)\,|\,D_1(\ell)$ as
\begin{align}
\Psi(z)\,|\,D_1(\ell)\equiv& \pa{\frac{g_{j-1}}{A_{\ell}^{2\ell^{j-2}}}\cdot E_{\ell-1}^{\ell^{j-1}}}\,|\,D_1(\ell)\llabel{summand1}\\
&\quad+\pa{g_{j-1}E_{\ell-1}^{\ell^{j-2}(\ell-1)} - \frac{g_{j-1}}{A_{\ell}^{2\ell^{j-2}}}}\,|\,D_1(\ell)\llabel{summand2}\\
&\quad+\pa{g_j(z)-g_{j-1}(z) E_{\ell-1}^{\ell^{j-2}(\ell-1)}}\,|\,D_1(\ell)\pmod{\ell^j}.\llabel{summand3}
\end{align}
It suffices to show that for each summand $S$, we have $S \ink{j} S_{\kl(\spt,j)}$. For~\eqref{summand1} and~\eqref{summand2} we will show this using properties of the trace operator and for~\eqref{summand3} we will use the standard filtration argument (Lemma~\ref{filtration-1}). Note that we have the congruences
\begin{equation}\llabel{congto1}
E_{\ell-1}^{\ell^{j-1}}\equiv A_{\ell}^{2\ell^{j-1}}\equiv 1\pmod{\ell^j}.
\end{equation}

Define
\begin{align*}
f(z):=\frac{g_{j-1}(z)}{A_{\ell}(z)^{2\ell^{j-2}}}\Dl \hspace{20pt} \and \hspace{20pt} 
h(z):=E_{\ell-1}(z)-\ell^{\ell-1}E_{\ell-1}(\ell z).
\end{align*}
By~(\ref{congto1}), the first summand~(\ref{summand1}) is
congruent to $ f$ modulo $\ell^j$. We will show that
\begin{equation}\llabel{trace-congruence}
f
\equiv 
\tr\pa{
fh^{\ell^{j-1}}
}
\pmod{\ell^j}.
\end{equation}
Since $\tr$ sends $M_{\kl(\spt,j)}^!(\Ga_0(\ell))$ to $M_{\kl(\spt,j)}^!$, and $f$ is cuspidal, this would show that \\$f\ink{j}S_{\kl(\spt,j)}$.
(Because the coefficients of the terms with nonpositive exponents in $\tr(fh^{\ell^{j-1}})$ are 0 modulo $\ell$, we can subtract from $\tr(fh^{\ell^{j-1}})$ a modular function to cancel out those terms, and this will differ from $fh^{\ell^{j-1}}$ by a multiple of $\ell^j$.)

To show~(\ref{trace-congruence}), we use Lemme 9 of~\cite{serre}, which gives
\[
\ord_{\ell}(\tr(fh^{\ell^{j-1}})-f)\ge \min(j+\ord_{\ell}(f),\ell^{j-1}+1+\ord_{\ell}(f\,|_2\,W(\ell))-1).
\]
The first argument is at least $j$; it suffices to show the second argument is at least $j$ as well. Note that if $F$ is a modular form modulo $\ell$ of weight $w$,
\[
F\,|\,U(\ell)=\ell^{-1}\sum_{k=0}^{\ell-1}F|_w\matt 1k0{\ell}.
\]
Letting $\ga_k=\matt{k\ell}{-1}{\ell^2}0$, we rewrite $h\,|_2\,W(\ell)$ as
\begin{align}
\nonumber f\,|_2\,W(\ell)&=
\ell^{-1}\sum_{k=0}^{\ell-1}\pa{\frac{g_{j-1}}{A_{\ell}^{2\ell^{j-2}}}\Phi_{\ell}}|_2\matt 1k0{\ell} |_2 \matt0{-1}{\ell}0\\
&=\ell\ell^{-1}\sum_{k=0}^{\ell-1}\frac{g_{j-1}|_{\kl(\spt,j-1)}\ga_k}{A_{\ell}^{2\ell^{j-2}}|_{\ell^{j-2}(\ell-1)}\ga_k}\Phi_{\ell}|_0\ga_k.\llabel{yumyum}
\end{align}
Note that the extra factor of $\ell$ in~(\ref{yumyum}), not present in~\cite{boylanwebb}, comes from combining the matrices after the slash operator. The rest of the proof is the same: for $k\ne 0$ we decompose the matrix $\ga$ as
\[
\ga_k=\matt{k}{b_k}{\ell}{k'}\matt{\ell}{-k'}{0}{\ell},
\]
where $k'$ is chosen so that $kk'\equiv 1\pmod{\ell}$, $b_k=\frac{kk'-1}{\ell}$. For $k=0$ we break up the matrix as
\[
\ga_0=\matt 100{\ell}\matt 0{-1}{\ell}0.
\]
Using transformation properties of $g_{j-1}$ and $\eta$, we
compute the following lower bounds for the valuation of the factors in the terms in~(\ref{yumyum}).
\[
\begin{tabular}{|c|c|c|c|}
\hline
&$g_{j-1}\,|_{k_{\ell}(\spt,j-1)}\ga_k$
&$(A_{\ell}^{2\ell^{j-2}}\,|_{\ell^{j-2}(\ell-1)}\,\ga_k)^{-1}$
&$\Phi_{\ell}\,|_0\,\ga_k$\\ \hline
$1\le k\le \ell-1$
&0 &0&$-\rc2$\\ \hline
$k=0$
&$k_{\ell}(\spt,j-1)$
&$-\ell^{j-1}$
&$-1$\\
\hline
\end{tabular}
\]
(The only change from~\cite{boylanwebb} is that in our case, $\kl(\spt,j-1)$ replaces $\ell^{j-2}(\ell-1)$.) 
From the above table and~(\ref{yumyum}) we get
\[
\ell^{j-1}+\ord_{\ell}(f|_2W(\ell))
\ge \ell^{j-1}+(k_{\ell}(\spt, j-1)-\ell^{j-1}-1)
= \ell^{j-1}-\ell^{j-2}+1\ge j,
\]
as needed.

For the second summand~\eqref{summand2}, note that by (\ref{congto1}), if we let 
\begin{equation}
\llabel{bjl}
B_{j,\ell}:=g_{j-1}-\frac{g_{j-1}}{A_{\ell}^{2\ell^{j-2}}}\cdot E_{\ell-1}^{\ell^{j-2}},
\end{equation}
we get that~(\ref{summand2}) is congruent to
$B_{j,\ell}\Dl \in M_{\kl(\spt,j-1)}^!(\Ga_0(\ell))$ modulo $\ell^j$.
Furthermore, we have \begin{equation}\llabel{congto0}
B_{j,\ell}\equiv 0\pmod{\ell^{j-1}}.
\end{equation}
Because $\Delta^{\frac{\ell^2-1}{24}}\equiv \Phi_{\ell}\pmod{\ell}$, we can replace $\Phi_{\ell}$ with $\Delta^{\frac{\ell^2-1}{24}}$ without changing $B_{j,\ell}$ modulo $\ell^j$.
We will show that $B_{j,\ell} \ss|\ss D_1(\ell) \ink{j}S_{\kld(\spt,j-1)}$,
where for convenience we set
\begin{equation}\llabel{kldj}
\kld(j) :=\kl(\spt, j) +\frac{\ell^2-1}{2} =\ell^{j-1}(\ell-1) +2 +\frac{\ell^2-1}{2}.
\end{equation}
Then, in light of~(\ref{congto0}), multiplying by a suitable power of $E_{\ell-1}$ will give that $B_{j,\ell} \ss|\ss D_1(\ell) \ink{j}S_{\kld(j-1)}$,
finishing the proof of this part.

First, note that if we define
\[
C_{j,\ell}:=\frac{g_{j-1}}{A_{\ell}^{2\ell^{j-2}}}\cdot E_{\ell-1}^{\ell^{j-2}}\De^{\frac{\ell^2-1}{24}},
\]
then by the definitions of $B_{j,\ell}$ and $D_1(\ell)$, we have
\begin{align}
\nonumber
B_{j,\ell}\Dl &\equiv g_{j-1}\De^{\frac{\ell^2-1}{24}}\,|\,U(\ell)-C_{j,\ell}\,|\,U(\ell)\\
\label{eq:bjl}
&\equiv g_{j-1}\De^{\frac{\ell^2-1}{24}}\,|\,U(\ell) -\ell^{\frac{\kld(j-1)}{2}-1}\tr(C_{j,\ell}\,|_{\kld(j-1)}\,W(\ell))\\
\nonumber
&\quad +\ell^{\frac{\kld(j-1)}{2}-1}C_{j,\ell}\,|_{\kld(j-1)}\,W(\ell)\pmod{\ell^j}.
\end{align}
Note that the first term is 
the reduction of a cusp form in $S_{\kld(j-1)}$. 
Since $A_{\ell}^2\in M_{\ell-1}(\Ga_0(\ell))$, we have $C_{j,\ell}\in M^!_{\kld(j-1)}(\Ga_0(\ell))$ and $\tr(C_{j,\ell}\,|_{\kld(j-1)}\,W(\ell))\in M_{\kld(j-1)}$. It remains to show that
\[
\ell^{\frac{\kld(j-1)}{2}-1}C_{j,\ell}\,|_{\kld(j-1)}\,W(\ell)\equiv 0\pmod{\ell^j}.
\]
By using transformation properties of $g_{j-1}$ and $\eta$ we can compute
\begin{align*}
&\ord_{\ell}\pa{\ell^{\frac{\kld(j-1)}{2}-1}C_{j,\ell}\,|_{\kld(j-1)}\,W(\ell)}\\
&\ge \frac{\kld(j-1)}{2}-1+\ord_{\ell}\ba{\pa{g_{j-1}E_{\ell-1}^{\ell^{j-2}}\De^{\frac{\ell^2-1}{24}}}\,|_{\kl(\spt, j-1)+\ell^{j-2}(\ell-1)+\frac{\ell^2-1}{2}}\,W(\ell)}\\
&\quad -\ord_{\ell}(A_{\ell}^{2\ell^{j-1}}\,|_{\ell^{j-2}(\ell-1)}\,W(\ell))\\
&=\frac{\kld(j-1)}{2}-1+\frac{\kl(\spt,j-1)+\ell^{j-2}(\ell-1)+\frac{\ell^2-1}{2}}{2}-\frac{\ell^{j-1}+\ell^{j-2}}2\ge j,
\end{align*}
as needed. (The only difference in~\cite{boylanwebb} is that $\kl(\spt,j-1)$ in the last line above is replaced by $\ell^{j-2}(\ell-1)$; this bound for the valuation is 1 greater than in \cite{boylanwebb}.)

Write the third summand (\ref{summand3}) as $F_{j,\ell} \Dl$, where $F_{j,\ell}$ is defined by $$F_{j,\ell}:=g_j-g_{j-1}E_{\ell-1}^{\ell^{j-2}(\ell-1)}.$$ Since $g_j\equiv g_{j-1}\pmod{\ell^{j-1}}$, we get the congruence 
\begin{equation}\llabel{fjl}
F_{j,\ell}\equiv 0\pmod{\ell^{j-1}}.
\end{equation}
We have that $\frac{F_{j,\ell}}{\ell^{j-1}}\De^{\frac{\ell^2-1}{24}} \in {S}_{\kld(j)}$,
so by a calculation using Lemma~\ref{filtration-1} we obtain 
\begin{equation}\llabel{bob}
\omega_\ell\pa{\frac{F_{j,\ell}}{\ell^{j-1}}\Dl} \leq (\ell-1)(\ell^{j-2}+1)+2.
\end{equation}
Multiplying by $\ell^{j-1}$ and an appropriate power of $E_{\ell-1}$ gives that $F_{j,\ell}\Dl \ink{j} S_{\kl(\spt,j)}$.
This finishes the proof of the first part.

\item The second part is similar except with the modified operator $Y_{r}(\ell)$ and the weights $\ell^{j-1}(\ell-1)$. Again we write $\Psi(z)\Dlr$ modulo $\ell^j$ as in~(\ref{summand1})--(\ref{summand3}). 
The $j=1$ case follows from Lemma~\ref{filtration-2}, so let $j\ge 2$.

To study the first summand~\eqref{summand1}, let $f=\fc{g_{j-1}}{A_{\ell}^{2\ell^{j-2}}}\Dlr$ and replace $f\,|_2\,W(\ell)$ by $f\,|_{0}\,W(\ell)$. We have
\[
f\,|_{0}\,W(\ell)=\ell^{-1} \sum_{k=0}^{\ell-1} \frac{g_{j-1}|_{\ell^{j-2}(\ell-1)}\ga_k}{A_{\ell}^{2\ell^{j-2}}\,|_{\ell^{j-2}(\ell-1)}\,\ga_k}\Phi_{\ell}^r\,|_0\,\ga_k.
\]
We have the following lower bounds on the valuations of the factors.
\[
\begin{tabular}{|c|c|c|c|}
\hline
&$g_{j-1}\,|_{\ell^{j-2}(\ell-1)}\ga_k$
&$(A_{\ell}^{2\ell^{j-2}}\,|_{\ell^{j-2}(\ell-1)}\,\ga_k)^{-1}$
&$\Phi_{\ell}^r\,|_0\,\ga_k$\\ \hline
$1\le k\le \ell-1$
&0 &0&$-\frac r2$\\ \hline
$k=0$
&$\ell^{j-2}(\ell-1)$
&$-\ell^{j-1}$
&$-r$\\
\hline
\end{tabular}
\]
(The change here is that the exponents coming from $\Phi_\ell$ are multiplied by $r$.)
Hence we get that
\begin{align}
\nonumber
\ord_{\ell}(f\ss|_{0}\ss W(\ell))& \ge 
-1+(\ell^{j-2}(\ell-1)-\ell^{j-1}-r)\\
&=-1-\ell^{j-2}-r.\llabel{f-pr}
\end{align}
From the inequality (Lemme 9 of~\cite{serre})
\[
\ord_{\ell}(\tr(fh^{\ell^{j-1}})-f)\ge \min\pa{j+\ord_{\ell}(f),\ell^{j-1}+1+\ord_{\ell}(f\ss|_{0}\ss W(\ell))}
\]
and from (using $\ell\ge r+3$)
\[
\ell^{j-1}+1+\ord_{\ell}(f\ss|_{0}\ss W(\ell))\ge \ell^{j-1}-\ell^{j-2}-r\ge j
\]
we obtain as before that $f\in^jS_{\ell^{j-1}(\ell-1)}$. Because $U(\ell)\equiv T(\ell)\pmod{\ell^m}$, we obtain $f\Ul\in^j S_{\ell^{j-1}(\ell-1)}$ as well.

For the second summand~\eqref{summand2}, note that replacing $D_1(\ell)$ by $D_r(\ell)$ means replacing $\De^{\fc{\ell^2-1}{24}}$ by $\De^{\fc{r(\ell^2-1)}{24}}$. Thus the proof is similar to that of Proposition 3.3 in~\cite{boylanwebb}, with $\fc{\ell^2-1}{2}$ replaced by $\fc{r(\ell^2-1)}2$. We get that $B_{j,\ell}\Dlr$ is the reduction of a cusp form of weight
\begin{equation}\label{pr-2-j>2}
\ell^{j-2} (\ell-1)+\fc{r(\ell^2-1)}2.
\end{equation}
When $j>2$, this is at most $\ell^{j-1}(\ell-1)$, since the inequality is equivalent to $\fc{r(\ell^2-1)}2\le \ell^{j-2}(\ell-1)^2$, which is true since $r\le \fc{2\ell(\ell-1)}{\ell+1}$. For $j=2$, note that we get $B_{j,\ell}\Dlr\in^2 S_{\ell-1+\fc{r(\ell^2-1)}2}$ and hence that
\[
B_{j,\ell}\Dlr\in^2 S_{\fl{\fc r2}\ell(\ell-1)}.
\]
By Lemma~\ref{filtration-1},
\[
\om_{\ell}\pa{\fc{B_{j,\ell}}{\ell}\Dlr\Ul}\le \ell+\fc{\fl{\fc r2}\ell(\ell-1)-1}{\ell}\le \ell(\ell-1),
\]
which shows, after multiplying by a suitable power of $E_{\ell-1}$, that $B_{j,\ell}\Ylr\in^2 S_{\ell(\ell-1)}$.

For the third summand~\eqref{summand3}, we calculate the filtration of $\frac{F_{j,\ell}}{\ell^{j-1}}\De^{\frac{r(\ell^2-1)}{24}}\,|\,U(\ell)$ using Lemma~\ref{filtration-1}. Since the filtration of $\frac{F_{j,\ell}}{\ell^{j-1}}$ is at most $\ell^{j-1}(\ell-1)$, we get that
\[\om_{\ell}\pa{
\frac{F_{j,\ell}}{\ell^{j-1}}\De^{\frac{r(\ell^2-1)}{24}}
\,|\,U(\ell)}\le
\ell+\frac{\ell^{j-1}(\ell-1)+\frac{r(\ell^2-1)}2-1} {\ell}.
\]
We also have that $\ell\ge r+3$ and the filtration is a multiple of $\ell-1$, so we get that
\begin{equation}
\om_{\ell}\pa{
\frac{F_{j,\ell}}{\ell^{j-1}}\De^{\frac{r(\ell^2-1)}{24}}\Ul
}\le  
(\ell-1)\pa{\ell^{j-2}+\frac{r+2}{2}} 
\le \ell^{j-1}(\ell-1). \llabel{pr-3-j>2}
\end{equation}
Another application of $U(\ell)$ will not increase the filtration.
\qedhere
\end{newlist}
\end{proof}

\begin{lemma} \llabel{bw-3.5} We have the following.
\begin{newlist} 
\item If $b \geq 1$ and $m \geq 1$, then we have that $L_\ell(spt,b;z)\ink{m} M_{k_\ell(spt,m)}$. (For $b>1$, $L_\ell(spt,b;z)$ is a cusp form.)
\item If $b \geq 0$ is even and $m \geq 1$, then we have that $L_\ell(r,b;z) \ink{m} M_{\ell^{m-1}(\ell-1)}$. (For $b>0$, $L_\ell(r,b;z)$ is a cusp form.) If $b\ge 1$ is odd and $m\ge 1$, then $L_{\ell}(r,b;z) \in^m M_{\klodd(r,m)}$.
\end{newlist}
\end{lemma}
\begin{proof} In both cases the proof is by induction on $b$. For (1), the base
case $b=1$ is Proposition~\ref{4.2}. For (2), the base case follows from $1\equiv E_{\ell-1}^{\ell^{j-1}}\pmod{\ell^j}$ for all $j$. 
In (1), the induction step for even $b$ to $b+1$ is given by Lemma~\ref{bw-3.1}(1), while the induction step for odd $b$ comes from the fact that $U(\ell)\equiv T(\ell)\pmod{\ell^j}$ as long as the weight is greater than $j$, and the latter operator preserves spaces of modular forms. In (2), the induction steps for even $b$ to $b+2$, and even $b$ to $b+1$, are given by Lemma~\ref{bw-3.1}(2).
\end{proof}
\begin{remarkstar} 
Let $M_r(\ell,m)$ and $M_r^{\text{odd}}(\ell, m)$ denote the space of modular forms in $M_{\ell^{m-1}(\ell-1)} \cap \zl$ and $M_{\klodd(r,m)}\cap \zl$, respectively, with coefficients reduced modulo $\ell^m$. The previous corollary shows that we have the following nesting of $\Z/\ell^m \Z$-modules in the case $p_r$:
\begin{gather*}
M_r(\ell,m) \contains \laeven(r,0,m) \contains \laeven(r,2,m) \contains \cdots \contains \laeven(r,2b,m) \contains \cdots\\
M_r^{\text{odd}}(\ell,m) \contains \laodd(r,1,m) \contains \laodd(r,3,m) \contains \cdots \contains \laodd(r,2b+1,m) \contains \cdots
\end{gather*}
In the case of $spt$, let $M_{spt}(\ell, m)$ denote the space of modular forms in $M_{k_\ell(\spt,m)}$ with coefficients reduced modulo $\ell^m$. Then we have the following inclusions:
\begin{gather*}
M_{\spt}(\ell,m) \contains \laeven(spt,2,m) \contains \laeven(spt,4,m) \contains \cdots \contains \laeven(spt,2b,m) \contains \cdots\\
M_{\spt}(\ell,m) \contains \laodd(spt,1,m) \contains \laodd(spt,3,m) \contains \cdots \contains \laodd(spt,2b+1,m) \contains \cdots
\end{gather*}
Since each sequence is contained in a finite-dimensional space, it must stabilize after a finite number of steps.
\end{remarkstar}

\subsection{Reducing the weight of modular forms}
The next lemma shows that $X_1(\ell)$ and $Y_r(\ell)$ reduce the weight
of those modular forms modulo $\ell^j$ in the $\Z/\ell\Z$--kernel of
$U(\ell)$ or $D(\ell)$. Think of the lemma as an analogue of Lemma~\ref{filtration-2} when we're considering $q$-series modulo higher powers of $\ell$.
\begin{lemma}
 \llabel{3.6}
 Let $n\ge 1$ and $\Psi(z) \in \Z[[q]]$. 
 \begin{newlist}
 \item Let $\ell \ge 5$ be prime. Suppose that, for all $1 \le j \le n$, we have $\Psi(z) \ink{j} M_{\kl(\spt,j)}$. If $\Psi(z) \Dl \equiv 0 \pmod \ell$, then for
all $2 \le j \le n$, we have $\Psi(z)\yspt \ink{j} S_{\kl(\spt,j-1)}$.
\item Let $\ell \geq r+5$ be prime. Suppose that $\Psi(z)\ink{1} M_{\ell-1}$ and for all $2 \le j
\le n$, we have $\Psi \ink{j} M_{\ell^{j-1}(\ell-1)}$. If $\Psi(z) \equiv 0 \pmod \ell$, 
then for all $2\le j \le n$, we have $\Psi(z)\Ylr \ink{j} S_{\ell^{j-2}(\ell-1)}$. 
 \end{newlist}
 \end{lemma}
 \begin{proof}
This proof follows the same argument as Lemma 3.6
in~\cite{boylanwebb}. We keep the notation from the proof of Lemma~\ref{bw-3.1}.
\begin{newlist} 
\item 
 We break up $\Psi(z)\yspt$ into three summands as follows:
 \begin{align} 
 \llabel{3.6-s1}\Psi(z)\yspt & \equiv \left(\frac{g_{j-1}}{A_{\ell}^{2\ell^{j-2}}} \cdot E_{\ell-1}^{\ell^{j-1}} \right)\yspt \\
 \llabel{3.6-s2}&\quad  + \left( g_{j-1} E_{\ell-1}^{\ell^{j-2} (\ell-1)} - \frac{g_{j-1}}{A_{\ell}^{2\ell^{j-2}}}\cdot E_{\ell-1}^{\ell^{j-1}} \right)\yspt \\
 \llabel{3.6-s3}&\quad + \left( g_j- g_{j-1} E_{\ell-1} ^{\ell^{j-2} (\ell-1)} \right)\yspt  \pmod {\ell^j}.
 \end{align}
We prove that for each summand $S$ we have $S \ink{j}S_{\kl(\spt,j-1)}$, and begin by assuming $j \geq 3$.

First summand~(\ref{3.6-s1}): By hypothesis, $0\equiv f(z)\Dl\equiv \frac{g_{j-1}}{A_{\ell}^{2\ell^{j-2}}}\Dl\pmod{\ell}$. Hence using $E_{\ell-1}(z)^{\ell^{j-2}}\equiv 1\pmod{\ell^{j-1}}$,~(\ref{3.6-s1}) is congruent to $G_{j,\ell}(z)\Ul$ modulo $\ell^j$, where
\[
G_{j, \ell} (z):= \left( \frac{g_{j-1}(z)}{A_{\ell}(z)^{2\ell^{j-2}}} \Dl \right) E_{\ell-1}(z)^{\ell^{j-2}} \in M_{\ell^{j-2}(\ell-1)+2}^{!} (\Gamma_0(\ell)).
\]
We have that
\begin{equation*}
\displaystyle \ell^{\frac{\ell^{j-2}(\ell-1)}{2}} \tr\left( G_{j,\ell}
 \,|_{k_{\ell}(\spt,j-1)}\,W(\ell) \right) = G_{j,\ell} \ss  | \, U(\ell) +
\ell^{\frac{\ell^{j-2}(\ell-1)}{2}} G_{j,\ell}
\,|_{k_{\ell}(\spt,j-1)}\,W(\ell). \end{equation*}
We will show that the valuation of the last term is at least $j$. Since $\tr$ sends $M_{\kl(\spt,j-1)}^!(\Ga_0(\ell))$ to $M_{\kl(\spt,j-1)}^!$, and $G_{j,l}(z)\Ul$ is cuspidal, this would show that $G_{j,l}(z)\Ul \ink{j} S_{\kl(\spt,j-1)}$, as needed.

Note that $\frac{g_{j-1}}{A_{\ell}^{2\ell^{j-2}}}\Dl\,|_{2}\,W(\ell)$ is given by~(\ref{yumyum}). Hence, we have 
\begin{align*}
&\ordl \left( \ell^{\frac{\ell^{j-2}(\ell-1)}{2}} G_{j,\ell} \,|_{k_{\ell}(\spt,j-1)}\,W(\ell) \right)\\
&\ge \frac{\ell^{j-2}(\ell-1)}{2}+\ordl\pa{E_{\ell-1}^{\ell^{j-2}} \,|_{\ell^{j-2}(\ell-1)}\,W(\ell)}+ \ordl \left(\frac{g_{j-1}}{A_\ell(z)^{2\ell^{j-2}}} \Dl \,|_{2}\,W(\ell) \right)\\
&
\ge \frac{\ell^{j-2}(\ell-1)}2 + \frac{\ell^{j-2}(\ell-1)}2 + (\kl(\spt, j-1)-\ell^{j-1}-1)\ge j.
\end{align*}

Second summand \eqref{3.6-s2}: 
Defining $B_{j,\ell}$ as in~(\ref{bjl}), 
we see that~\eqref{3.6-s2} is congruent to $B_{j,\ell}\,|\,Y_1(\ell)$ modulo $\ell^j$.
From~(\ref{kldj}) and Lemma~\ref{filtration-1} we get
\begin{align*}
\om_{\ell}\pa{\frac{B_{j,\ell}}{\ell^{j-1}}\Dl}&\le\ell^{j-2}(\ell-1)+2+\frac{\ell^2-1}{2}, \text{ so }\\
\om_{\ell}\pa{\frac{B_{j,\ell}}{\ell^{j-1}}\Dl\Ul}&\le\ell+\frac{
\ell^{j-2}(\ell-1)+2+\frac{\ell^2-1}{2}
-1}{2}\le \kl(\spt,j-1).
\end{align*}
Multiplication by an appropriate power of $E_{\ell-1}$ then shows  $B_{j,\ell}(z)\,|\,Y_1(\ell) \ink{j} S_{\kl(\spt, j-1)}$.
 
Third summand~(\ref{3.6-s3}): 
Defining $F_{j,\ell}$ as in~(\ref{fjl}), 
we find that~(\ref{3.6-s3}) is congruent to $F_{j,\ell} \Yl$ modulo $\ell^j$. Using~(\ref{bob}), Lemma~\ref{filtration-1} gives 
\[
\om_\ell \pa{\frac{F_{j,\ell}}{\ell^{j-1}}\Dl\Ul}\le \ell+\frac{(\ell-1)(\ell^{j-2}+1)+2-1}{\ell}\le \kl(\spt,j-1).
\]
We finish with the same argument as before.

\item 
First consider $j>2$. 
As in the first part, we decompose $\Psi(z) \Ylr$ into three
summands, each of which we will show to be the reduction of a form in
$S_{\ell^{j-2}(\ell-1)}$:
\begin{align} \Psi \Ylr &\equiv
\pa{\frac{g_{j-1}}{A_\ell^{2\ell^{j-2}}} \cdot
E_{\ell-1}^{\ell^{j-1}}} \Ylr \llabel{3.6.2-summand1}
\\& \quad+ \pa{g_{j-1}E_{\ell-1}^{\ell^{j-2}(\ell-1)} -
\frac{g_{j-1}}{A_\ell^{2\ell^{j-2}}}\cdot
E_{\ell-1}^{\ell^{j-1}}} \Ylr\llabel{3.6.2-summand2}
\\& \quad+ \pa{g_j - g_{j-1}E_{\ell-1}^{\ell^{j-2}(\ell-1)}} \Ylr .\llabel{3.6.2-summand3}
\end{align}

First summand~(\ref{3.6.2-summand1}): Letting $G_{j,\ell}(z) =\pa{\frac{g_{j-1}(z)}{A_\ell(z)^{2\ell^{j-2}}}\Dlr }
E_{\ell-1}(z)^{\ell^{j-2}} \in M^!_{\ell^{j-2}(\ell-1)}(\Ga_0(\ell))$ and recalling $E_{\ell-1}^{\ell^{j-2}}\equiv 1\pmod{\ell^{j-1}}$, we find that the first summand is congruent modulo $\ell^j$ to
$G_{j,\ell}(z) \Ul$.
As before, considering
\[
\ell^{\frac{\ell^{j-2}(\ell-1)}{2}-1}\tr(G_{j,\ell}(z) \Ul) = G_{j,\ell}(z)
\Ul + \ell^{\frac{\ell^{j-2}(\ell-1)}{2}-1}
G_{j,\ell}(z)\,|_{\ell^{j-2}(\ell-1)}\,W(\ell),
\]
it suffices to show the valuation of the last term is at least $j$. Using~\eqref{f-pr}, we have that 
\begin{align*}
&\ord_\ell\pa{\ell^{\frac{\ell^{j-2}(\ell-1)}{2}-1}G_{j,\ell}(z)\,|_{\ell^{j-2}(\ell-1)}\,W(\ell)}\\
&=\pa{\frac{\ell^{j-2}(\ell-1)}2-1} + \ordl\pa{\frac{g_{j-1}}{A_{\ell}^{2\ell^{j-2}}}\Dl\,|_{0}\,W(\ell)} +\ord_{\ell} \pa{ E_{\ell-1}^{\ell^{j-2}}\,|_{\ell^{j-2}(\ell-1)}\,W(\ell)}\\
&\ge \pa{\frac{\ell^{j-2}(\ell-1)}{2}-1} + \left(-1 -
\ell^{j-2}-r\right)+\frac{\ell^{j-2}(\ell-1)}{2}
\geq j.
\end{align*}

%
The second summand $\eqref{3.6.2-summand2}$ is $B_{j,\ell} \Dlr \Ul$,
where we define $B_{j,\ell}$ as in~\eqref{bjl}.
From~\eqref{pr-2-j>2} it follows that 
$B_{j,\ell}\Dlr \ink{j} S_{\ell^{j-2}(\ell-1) + \frac{r(\ell^2-1)}{2}}$. 
Lemma~\ref{filtration-1} and the fact that $r + 2 < \ell$ then give
\[
\omega_\ell\pa{\frac{B_{j,\ell}}{\ell^{j-1}}\Dlr \Ul} \leq \ell +
\frac{\ell^{j-2}(\ell-1)+\frac{r(\ell^2-1)}{2}-1}{\ell}
\le \ell^{j-2}(\ell-1).
\]
Since $(B_{j,\ell} \Dlr) \Ul\equiv f_{j,\ell}\Ul\equiv 0\pmod{\ell^{j-1}}$, we may multiply by an appropriate power of $E_{\ell-1}$ to show $(B_{j,\ell} \Dlr) \Ul \ink{j} S_{\ell^{j-2}(\ell-1)}$ modulo $\ell^j$. 

%
For the third summand, we define $F_{j,\ell}$ as in~\eqref{fjl}. From~\eqref{pr-3-j>2} we have that
${F_{j,\ell}(z)}\Dlr$ is congruent modulo $\ell$ to a
form in $S_{(\ell^{j-2}+\frac r2+1)(\ell-1)}$. Again using Lemma
\ref{filtration-1}, we get
\[
\omega_\ell\pa{\frac{F_{j,\ell}}{\ell^{j-1}}\Dlr \Ul} \leq \ell + \frac{(\ell^{j-2}+\frac r2+1)(\ell-1)-1}{\ell}
\le \ell^{j-2}(\ell-1).
\]
Exactly as before, multiplication 
by a power of $E_{\ell-1}$ gives that $F_{j,\ell}(z) \Dlr \Ul \ink{j} S_{\ell^{j-2}(\ell-1)}$.

To check the $j=2$ case, note that by hypothesis, 
we have $\om_{\ell}\pf{g_2}{\ell}\le\ell(\ell-1)$. Applying Lemma~\ref{filtration-1} repeatedly and noting $2\ell> r+3$ gives
\begin{align*}
\om_{\ell}\pa{\frac{g_2\Dlr}{\ell}}&\le (\ell-1)\pa{\frac{r+1}2},\\
\om_{\ell}\pa{\frac{g_2\Dlr}{\ell}\Ul}&\le \ell-1.
\end{align*}
Hence $\Psi\Yl\equiv g_2\Dlr\Ul\pmod{\ell^2}$ is the reduction of a form in $S_{\ell-1}$ modulo $\ell^2$.\qedhere
\end{newlist}
\end{proof}

\section{The spaces $\olmodd(m)$ and $\olmeven(m)$}
\label{S4}

In this section, we will give injections from our stabilized spaces into spaces of cusp forms of small weight.  

Following \cite{boylanwebb}, we recall two elementary commutative algebra results.  
\begin{lemma}
Let $A$ be a finite local ring, $M$ be a finitely generated $A$-module, and $T: M \to M$ be an $A$-isomorphism.   
\begin{newlist}
\item There exists an integer $n > 0$ such that $T^n$ is the identity map on $M$. 
\item For all $\mu \in M$ and $n \geq 0$, we have 
$$\mu \in A[T^n(\mu), T^{n+1}(\mu), \ldots].$$
\end{newlist}
\end{lemma}

Now, we will state our main theorem,  which we will need a few lemmas to prove.  

\begin{theorem}
\label{injection}
The following hold.
\begin{newlist}
\item If $\ell \geq 5$ is prime and $m \geq 1$, then there exists an injective $\Z/\ell^m\Z$--module homomorphism 
$$\piodd(\spt): \olmodd(\spt, m) \hr S_{\ell +1} \cap \Z_{(\ell)}[[q]]$$ satisfying the following property: for all $\nu \in \olmodd(\spt, m)$ with $\ord_\ell(\nu) = j < m$, we have 
$$\piodd(\spt)(\nu) \equiv \nu \pmod{\ell^{j+1}}.$$
\item If $\ell \geq r+5$ is prime and $m \geq 1$, then there exist injective $\Z/\ell^m\Z$--module homomorphisms
$$\pieven(r): \olmeven(r, m) \hr S_{\ell-1} \cap \Z_{(\ell)}[[q]],\qquad \piodd(r):\olmodd(r,m)\hra S_{\klodd(r,1)}\cap \zl$$
satisfying the following property: for all $\mu_1 \in \olmeven(r, m)$ and $\mu_2\in \olmodd(r,m)$ with $\ord_\ell(\mu_i) = j_i < m$, we have 
$$\pieven(r)(\mu_1) \equiv \mu_1 \pmod{\ell^{j_1+1}},\qquad \piodd(r)(\mu_2) \equiv \mu_2 \pmod{\ell^{j_2+1}}.$$
\end{newlist}
\end{theorem}

We will work out the first case in detail, and state results for the second only when they differ significantly from the $r = 1$ case in \cite{boylanwebb}.   

For the $spt$ case, we consider the following two submodules of $S_{\kl(\spt, m)} \cap \Z_{(\ell)}[[q]]$:
\begin{align}
\cals_0(\spt,m) &:= \bc{f(z) E_{\ell - 1}(z)^{\ell^{m - 1} - 1}\st f(z) \in S_{\ell + 1} \cap \Z_{(\ell)}[[q]]} \and\\
\cals_1(\spt,m) &:= \bc{g(z) \st g(z) = \sum_{j = m_0}^\infty a_j q^j \in S_{\kl(\spt, m)} \cap \Z_{(\ell)}[[q]] \text{ with } m_0 > \dim(S_{\ell+1})}.
\end{align}

By the existence of ``diagonal bases" (with integer Fourier coefficients) for spaces of cusp forms, 
we have 
$S_{\kl(\spt, m)} \cap \Z_{(\ell)}[[q]] = \cals_0(\spt,m) \oplus \cals_1(\spt,m)$. 
We define $\sodd(\spt,m) \subeq S_{\kl(\spt, m)} \cap \Z_{(\ell)}[[q]]$ to be the largest $\Z/\ell^m\Z$--submodule such that $X_1(\ell)$ is an isomorphism on $\sodd(\spt,m)$ modulo $\ell^m$. 

For $p_r$, we consider the following two submodules of $S_{\ell^{m-1}(\ell-1)} \cap \Z_{(\ell)}[[q]]$: 
\begin{align}
\seven_0(r,m) &:= \bc{f(z)E_{\ell - 1}(z)^{\ell^{m-1}-1}: f(z) \in S_{\ell-1}\cap \zl} \and \\
\seven_1(r,m) &:= \bc{g(z): g(z) = \sum_{j = m_0}^{\infty} a_jq^j \in S_{\ell^{m-1}(\ell-1)} \text{ with } m_0 > \dim(S_{\ell-1})}.
\end{align}
We similarly see that $S_{\ell^{m-1}(\ell-1)} \cap \zl= \seven_0(r) \oplus\seven_1(r,m)$, and define $\seven(r)$ to be  the largest $\Z/\ell^m \Z$--submodule such that $Y_r(\ell)$ is an isomorphism on $\seven(r,m)$ modulo $\ell^m$. Define $\sodd(r,m)$ similarly but use $\klodd(r,m)$ in place of $\ell^{m-1}(\ell-1)$ and with $\fc{\klodd(r,m)-\klodd(r,1)}{\ell-1}$ in place of $\ell^{m-1}-1$.

The following lemma holds for both $\spt$ and $p_r$, and for both $\seven$ and $\sodd$. Therefore, for brevity, we shall let $\cals=\sodd(\spt,m)$, $\seven(r,m)$, or $\sodd(r,m)$, and let $\cals_0$, and $\cals_1$ denote the corresponding spaces.
\begin{lemma}
\label{4.4}
Suppose that $f(z) \in \cals$ has $\ord_\ell(f) = i < m$, and that $f(z) = f_0(z) + f_1(z)$ with $f_i(z) \in \cals_i$.  Then we have $\ord_\ell(f_1) > i$ and $\ord_\ell(f_0) = i$.  
\end{lemma}

\begin{proof}
Using Lemma~\ref{filtration-2}, this proof proceeds as Lemma 4.4 of \cite{boylanwebb}.  
\end{proof}

Now, we state a corollary of Lemma \ref{4.4}, whose proof follows Corollary 4.7 of \cite{boylanwebb} exactly.  
\begin{corollary}
\label{4.7}
Let $f(z), g(z) \in \cals$, and suppose that $f(z) = f_0(z) + f_1(z)$ and $g(z) = g_0(z) + g_1(z)$ with $f_i, g_i \in \cals_i$.  Suppose further that $f_0(z) \equiv g_0(z) \pmod{\ell^m}$.  Then we have $f(z) \equiv g(z) \pmod{\ell^m}$.  
\end{corollary}

We are now ready to construct our injection and prove Theorem \ref{injection}. 

\begin{proof}[Proof of Theorem \ref{injection}]
We construct the injection $\piodd(\spt)$ as the composition of three $\Z/\ell^m\Z$-module homomorphisms $\Psi_1(\spt), \Psi_2(\spt)$, and $\Psi_3(\spt)$.  
By Lemma \ref{bw-3.5}, we have that $X_1(\ell)$ is an isomorphism on
$\olmodd(\spt, m)$, which implies that $\olmodd(\spt, m) \subeq
\sodd(\spt,m)$ by the definition of $\sodd(\spt,m)$.  Therefore,
we let the map $\Psi_1(\spt)$ be defined by 
$$\Psi_1(\spt): \olmodd(\spt, m) \hr
\sodd(\spt,m).$$

To define $\Psi_2(\spt)$, we let $f(z) = f_0(z) +
f_1(z) \in \sodd(\spt,m)$ with $f_i \in \cals_i(\spt,m)$, and we
suppose that $\ord_\ell(f) < m$.  Lemma \ref{4.4} implies that $f(z)
\equiv f_0(z) \pmod{\ell^{\ord_\ell(f) + 1}}$.  Therefore, we can let
$\Psi_2(\spt): \sodd(\spt,m) \to \cals_0(\spt,m)$ be defined by
$$\Psi_2(\spt): f(z) \mapsto f_0(z) \pmod{\ell^{\ord_\ell(f) + 1}}.$$  This
map is injective by Corollary \ref{4.7}.  

We next define the map
$\Psi_3: \cals_0(\spt,m) \to S_{\ell + 1} \cap \Z_{(\ell)}[[q]]$.  Suppose
that $f(z) \in \cals_0(\spt,m)$.  By the definition of $\cals_0(\spt,m)$,
there exists $g(z) \in S_{\ell + 1} \cap \Z_{(\ell)}[[q]]$ such that
$f(z) = g(z)E_{\ell - 1}(z)^{\ell^{m - 1} - 1}$.  We therefore define
$$\Psi_3(\spt): f(z) \mapsto g(z).$$  

Since the first two are injections
and the third is an isomorphism, the composition $\piodd(\spt)$ is an
injection.  

Moreover, if we suppose that $f(z) \in \olmodd(\spt, m)$ with $\ord_\ell(f) < m$, then we have that 
$$\piodd(\spt)(f(z)) \equiv f(z) \pmod{\ell^{\ord_\ell(f) + 1}}.$$

This proves the theorem for the $spt$ case.  The homomorphisms $\Psi_1(r), \Psi_2(r), \and \Psi_3(r)$, whose composition gives us $\pieven(r)$ or $\piodd(r)$, are defined similarly.  
\end{proof}

\begin{remarkstar}
These injections preserve the order of vanishing.  Since applying $D_r(\ell)$ to a form gives $q$-expansions satisfying 
$$F  \ss| \ss D_r(\ell) = \sum_{n \ge \frac{r(\ell^2 - 1)}{24 \ell}} a_nq^n,$$
this gives us that 
$$\text{rank}_{\Z/\ell^m \Z} \pa{ \olmodd(\spt, m)} \leq \dim \pa{S_{\ell + 1}} - \fl{\frac{\ell^2 - 1}{24 \ell}}.$$

For the $p_r$ case, note that we have isomorphisms $D_{\ell}(r):\seven(r,m)\xrightarrow{\cong} \sodd(r,m)$ and $U(\ell): \sodd(r,m) \xrightarrow{\cong} \seven(r,m)$. Hence
\[
\text{rank}_{\Z/\ell^m \Z} \pa{ \olmeven(r, m)}=\text{rank}_{\Z/\ell^m \Z} \pa{ \olmodd(r, m)} \leq \dim\pa{S_{\klodd(r, 1)}} - \fl{\frac{r(\ell^2 - 1)}{24 \ell}}.
\]
\end{remarkstar}

\section{Bounds on $\bl(r,m)$ and $\bl(\spt,m)$}
\llabel{S5}
In this section, we prove that under certain assumptions, we can give bounds for $\bl(r,m)$ and $\bl(\spt,m)$ of size roughly $2m$. We will first introduce some notation.

 Let $\sspt:=\cals(\spt,1)$ denote the largest subspace of $S_{\ell+1} \cap \Z_{(\ell)} [[q]]$ over $\Z/\ell \Z$ on which $X_{1}(\ell)$ is an isomorphism. As in \cite{boylanwebb}, we define $\dspt$ by
 \begin{equation}
 \dspt:= \min\{t \ge 0 : \, \, \forall f \in M_{\ell+1} \cap \Z_{(\ell)} [[q]], f \,|\,  D_1(\ell) \,|\,  X_{1}(\ell)^t \in \sspt \}.\llabel{dspt}
 \end{equation}
 Similarly define $\seven(r):=\seven(r,1)$ to be the largest subspace of $S_{\ell-1} \cap \zl$ over $\Z/\ell \Z$ on which $X_{r}(\ell)$ is an isomorphism and $\sodd(r):=\sodd(r,1)$ to be the largest subspace of $S_{\klodd(r,1)}$ on which $Y_r(\ell)$ is an isomorphism. 
 We define
 \begin{equation}
  \dr:= \min\{t \ge 0 : \, \, \forall f \in M_{\ell-1} \cap \Z_{(\ell)} [[q]], f \,|\,  D_r(\ell) \,\Xl^t \in \sodd(r) \}.\llabel{dr}
 \end{equation}
We also define a related quantity
\begin{equation}
  d_{\ell}'(r):= \min\{t \ge 0 : \, \, \forall f \in M_{\ell-1} \cap q^{\ce{\frac{r(\ell^2-1)}{24\ell^2}}}\Z_{(\ell)} [[q]], f \,|\,  Y_r(\ell)^t \in \seven(r) \}.\llabel{drp}
\end{equation} 
By the fact that we have isomorphisms $D_{\ell}(r):\seven(r)\xrightarrow{\cong} \sodd(r)$, $U(\ell): \sodd(r) \xrightarrow{\cong} \seven(r)$ and the fact that $\Phi_{\ell}^r\Ul\Ul$ only has terms with exponents at least $\ce{\frac{r(\ell^2-1)}{24\ell^2}}$, we have that $d_{\ell}(r)\le d_{\ell}'(r)+1$ and $d_{\ell}'(r)\le d_{\ell}(r)+1$.
We conjecture that $\dr=0$ for all $\ell\ge r+5$. The same is not true of $d_{\ell}'(r)$; however in many cases we still have $d_{\ell}'(r)=0$.

 We will prove the following theorem, which gives the last component of our main theorems. 
  \begin{theorem}
 \llabel{5.1}
 If $m\ge 1$, 
then the following bounds hold. 
 \begin{newlist}
 \item For $\ell \geq 5$ prime, we have \begin{equation*} 
 \bl(\spt,m) \le 2(\dspt+1)m+1.
 \end{equation*}
 \item For $\ell \geq r + 5$ prime, we have
 \begin{align*}
 \bl(r,1) &\le 2\dr+1,\\
 \bl(r,m) &\le 2(\dr+1)+2(d_{\ell}'(r)+1)(m-1),\qquad m\ge 2
 \end{align*}
 \end{newlist}
 \end{theorem}
\begin{remarkstar}
 Notice that if $\dspt=0$ and $d_{\ell}(r)=d_{\ell}'(r)=0$, then we get that
\[ \bl(\spt,m) \le 2m+1 \text{ and } \bl(r,m) \le 2m.\]
\end{remarkstar}
 
 Our proof will be similar to the proof of Theorem 5.1 in \cite{boylanwebb}. 
 We will need the following lemmas which we will use in the proof of Theorem~\ref{5.1}.
 \begin{lemma}
 \llabel{ceva}
  The following hold.  
  \begin{newlist}
 \item Suppose that, for some odd $b\ge 0$, we have $f(z) \in \laodd (\spt,b,m)$ and $0\le \ordl(f)<m$. If, moreover, we assume that there exists $g(z) \in \ml$ such that
 $$ f(z) \equiv \ell^{\ordl(f)} g(z) \pmod {\ell^{\ordl(f)+1}},$$
 then there exists $h(z) \in \olmeven(\spt,m)$ such that
 $$ f(z) \Dl \,|\,  X_{1}(\ell)^{\dspt} \equiv h(z) \pmod {\ell^{\ordl(f)+1}}.$$
 
 \item Suppose that, for some even $b\ge 1$, we have $f(z) \in \laeven (r,b,m)$ and $0\le \ordl(f)<m$. If, moreover, we assume that there exists $g(z) \in M_{\ell-1} \cap \zl$ such that
 $$ f(z) \equiv \ell^{\ordl(f)} g(z) \pmod {\ell^{\ordl(f)+1}},$$
 then there exists $h(z) \in \olmodd(r,m)$ such that
 $$ f(z) \,\,  \Dlr \,\,  \Xl^{\dr} \equiv h(z) \pmod {\ell^{\ordl(f)+1}}.$$
 
 \end{newlist}
 \end{lemma}
 \begin{proof}
 See the proof of Lemma 5.2 in \cite{boylanwebb}.
 \end{proof}
 We will also need the following lemma about the spaces $\olmeven(\spt,1)$ and $\olmodd(r,1)$.
 \begin{lemma}
 \llabel{5.3}
$\,$
 \begin{newlist}
 \item  If $\ell\ge 5$ is prime, then we have that 
$$ \Ll (\spt,2\dspt+3;z) \pmod \ell \in \olmodd(\spt,1).$$
 \item If $\ell\ge r+5$ is prime, then we have that
$$\Ll(r,2\dr+1;z) \pmod \ell \in \olmodd(r,1).$$
 \end{newlist}
 \end{lemma}
 \begin{remarkstar}
 Notice that the lemma gives the case $m=1$ of Theorem~\ref{5.1}.
 \end{remarkstar}
 \begin{proof}
\begin{newlist} 
\item Let $$f(z)=\Ll (\spt,1;z) \in \laodd(\spt,1,1).$$
Notice that if $\alpha_\ell(z) \equiv 0 \pmod \ell$, then $\Ll(\spt,b;z) \equiv 0 \pmod \ell$, and the conclusion holds trivially. 

Therefore, we may assume $\alpha_\ell(z) \nequiv 0 \pmod \ell$, which implies that $\ordl(f)=0$. By Th\'eor\`eme 11 on p. 228 of \cite{serre}, we have that $f(z)\ink{1} M_{\ell+1}$. By Lemma~\ref{ceva} we obtain $h(z) \in \olmodd(\spt,m)$ such that 
 \begin{equation*}
 f(z) \,|\,  U(\ell) \,|\, D_1(\ell)\,|\,  Y_{1}(\ell)^{\dl}= \Ll(\spt,2 \dspt+3;z) \equiv h(z) \pmod \ell.\end{equation*}
 Since reduction modulo $\ell$ maps $\olmodd(\spt,m)$ to $\olmodd(\spt, 1)$, the conclusion now follows. 
 \item
 The proof is similar to that of the first part of this lemma. We start with $\Ll(r,0;z)=1$, use Lemma~\ref{ceva}, and argue as above.\qedhere
 \end{newlist}
 \end{proof}
 
We use the next lemma in the proof of Lemma \ref{5.5}.
 \begin{lemma} \llabel{5.4}
The following hold.  
 \begin{newlist}
\item If $\ell\ge 5$ is prime and $f(z) \in \sodd(spt,m)$ is such that $0 \le \ordl(f)<m$, then for all $1\le s \le m-\ordl(f)$, we have that $f(z) \ink{\ordl(f)+s} S_{\kl(\spt,s)}$.
\item If $\ell\ge r+5$ is prime and $f(z) \in \seven(r,m)$ is such that $0 \le \ordl(f)<m$, then for all $1\le s \le m-\ordl(f)$, we have that $f(z) \ink{\ordl(f)+s} S_{\ell^{s-1}(\ell-1)}$. 
 \end{newlist}
 \end{lemma}
 \begin{proof}
 \begin{newlist}
 \item
See the proof of Lemma 5.3 in \cite{boylanwebb}. We use induction on $s$, the case $s=1$ being Lemma~\ref{4.4}. The proof goes through the same way because $\kl(\spt,s+1)-\kl(\spt,s)$ equals $\kl(s+1)-\kl(s)$ as defined in~\cite{boylanwebb}.
 \item The proof follows the same outline as in~\cite{boylanwebb}, but for the even rather than the odd space.
\end{newlist}
 
 
 \end{proof}
The following lemma is the last lemma required in the proof of Theorem~\ref{5.1}. Multiple applications of $U(\ell)$ and $D_r(\ell)$ to $L_{\ell}(\spt,b;z)$ and $L_{\ell}(r,b;z)$ form a finite sequence that $\ell$-adically approaches elements of $\olmodd(spt, m)$ and $\olmeven(r,m)$, respectively; we count the number of steps this process takes.
 \begin{lemma} \llabel{5.5}
 The following hold for $m\ge 2$.
 \begin{newlist}
 \item Suppose $\ell\ge 5$ is prime. If $1\le s \le m-1$, then there exist $f_{1,s}(z)=f_1(2(d_{\ell}(\spt)+1)s+1;z) \in \olmeven(\spt,m)$ and $f_{2,s}(z)=f_2(2(d_{\ell}(\spt)+1)s+1;z) \in S_{\kl(\spt,m-s)} \cap \zl$ such that the following properties hold.\smallskip
 \begin{newlist2}
 \item We have that $f_{2,s}(z) \equiv 0 \pmod {\ell^s}$.
 \item For all $k$ with $s+1 \le k \le m$, we have that $f_{2,s}(z) \ink{k} S_{\kl(\spt,k-s)}$.
 \item We have that $\Ll (\spt,2(\dspt+1)s+2;z) \equiv f_{1,s}(z) +f_{2,s}(z)  \pmod {\ell^m}.$
 \end{newlist2}  
 \item Let $r$ be given and suppose $\ell\ge r+5$ is prime. If $0\le s \le m$, then there exist $f_{1,s}(z)  \in \olmeven(r,m)$ and $f_{2,s}(z)  \in S_{\ell^{m-s-1}(\ell-1)} \cap \zl$ such that the following properties hold.\smallskip
 \begin{newlist2}
 \item We have that $f_{2,s}(z)  \equiv 0 \pmod {\ell^s}$.
 \item For all $k$ with $s+1 \le k \le m$, we have that $f_{2,s}(z)  \ink{k} S_{\ell^{k-s}(\ell-1)}$.
 \item We have that $\Ll (r,2(\dr+1)+2(\drp+1)(s-1);z) \equiv f_{1,s}(z) +f_{2,s}(z)  \pmod {\ell^m}.$
 \end{newlist2}  
 \end{newlist} 
 \end{lemma}
 \begin{proof}
In both parts we follow Lemma 5.5 in \cite{boylanwebb} and proceed by induction on $s$.
 
 \begin{newlist}
 \item We will only sketch the base case, since the induction step follows exactly as in \cite{boylanwebb}.
 
 Using Lemma~\ref{5.3}, we get a form $g(z) \in \olmodd(\spt,m)$ with $\Ll(\spt,2\dspt+3;z) \equiv g(z) \pmod \ell.$
 Since $U(\ell): \olmeven(\spt,m) \to \olmodd(\spt,m)$ is a bijection, it follows that there exists $f_1(2 \dspt +2; z) \in \olmeven(\spt,m)$ such that $$f_1(2 \dspt +2; z) \Ul \equiv g(z) \pmod{\ell^m}.$$

Now we will apply Lemma~\ref{3.6} to $\Ll(\spt,2\dspt+2;z)-f_1(2\dspt+2;z)$. It is easy to check that the hypotheses of Lemma~\ref{3.6} are satisfied. 
 Thus we get, for every $2\le k \le m$, a form $h_k(z) \in S_{
\kl(\spt,k-1)
}$ such that
 $$\left(\Ll(\spt,2\dspt+2;z)-f_1(2\dspt+2;z)\right) \Xl\equiv h_k(z) \pmod {\ell^k}. $$
 
 If we define $f_2(2(\dspt+1)+2;z) := h_m(z)$, then $f_2(2(\dspt+1)+2;z) \equiv 0 \pmod \ell$, and $f_2(2(\dspt+1)+2;z) \ink{k} S_{\ell^{k-s-1}(\ell-1)+2}$. Moreover, we have
 \begin{align*}
 \Ll(\spt,2\dspt+4;z) =& f_1(2\dspt+2;z) \yspt \\
 & + (\Ll(\spt,2\dspt+2;z) - f_1(2\dspt+2;z)) \yspt \\
  \equiv& f_1(2 \dspt+4;z)+f_2 (2\dspt+4;z) \pmod{\ell^m},
 \end{align*} which gives the conclusion.
 \item 
%
By Lemma~\ref{5.3} we get a form $g(z)\in \olmodd(r,m)$ with $L_{\ell}(r,2\dr+1;z)\equiv g(z)\pmod{\ell}$. Let
\begin{align*}
f_{1,1}(z)&=g(z)\Ul\in \olmeven(r,m),\\
f_{2,1}(z)&=L_{\ell}(r,2(\dr+1);z)-g(z)\Ul.
\end{align*}
This shows the base case.

For the induction step, assume the lemma true for $s-1$; we show it holds for $s$. 
By Lemma~\ref{3.6} applied to $\rc{\ell^{s-1}}f_{2,s-1}$, we have that one application of $Y_{\ell}(r)$ decreases the weight:
\[
f_{2,s-1}\Ylr \in^k S_{\ell^{k-s}(\ell-1)} ,\qquad s\le k\le m.
\]
By definition of $\drp$ we obtain that $\drp$ more applications of $Y_{\ell}(r)$ brings $f_2$ into the stabilized space:
\[
f_{2,s-1}\Ylr^{\drp+1}\in^s\ell^{s-1}\olmeven(r,1).
\]
Suppose $f_{2,s-1}\equiv \be \pmod{\ell^s}$ where $\be\in \ell^{s-1}\olmeven(r,m)$ and let
\begin{align*}
f_{1,s}&=f_{1,s-1}\Ylr^{\drp+1}+\be\\
f_{2,s}&=f_{2,s-1}\Ylr^{\drp+1}-\be.
\end{align*}
This completes the induction step.
\end{newlist}
\end{proof}
 We will now return to the proof of Theorem~\ref{5.1}. We use Lemma~\ref{5.5} along with a similar argument for going from $m-1$ to $m$.
 \begin{proof}[Proof of Theorem~\ref{5.1}]
 The proof being similar to that in~\cite{boylanwebb}, we will give full details for the $\spt$ case, and for the second statement we will only sketch the details.
\begin{newlist}
\item 
We will show that
\begin{equation*}
\Ll(\spt,2(\dspt+1)m+1;z) \in \olmodd(\spt,m). \end{equation*}
The case $m=1$ is Lemma~\ref{5.3}(1). Now suppose $m>1$. Using Lemma~\ref{5.5} with $s=m-1$, we get $f_1(2(\dspt+1)(m-1)+2;z) \in \olmeven(\spt,m)$ and $f_2(2(\dspt+1)(m-1)+2;z) \in S_{\ell+1} \cap \zl$ satisfying the properties in Lemma~\ref{5.5}. We have $$f_2 (2(\dspt+1)(m-1)+2;z) \equiv 0 \pmod {\ell^{m-1}}.$$
 If $f_2 (2(\dspt+1)(m-1)+2;z) \equiv 0 \pmod {\ell^{m}}$, then we get that $$ \Ll (\spt,2(\dspt+1)(m-1)+2;z) \equiv f_1(\spt,2(\dspt+1)(m-1)+2;z) \pmod {\ell^m},$$ so $\Ll (\spt,2(\dspt+1)(m-1)+2;z) \in \olmeven(\spt,m)$. 
 
 Since we have
  \begin{equation*}
 \Ll (\spt,2(\dspt+1)m+1;z)= \Ll (\spt,2(\dspt+1)(m-1)+2;z)  \,|\,  D_1(\ell) \,|\, X_{1}(\ell)^{\dspt},
 \end{equation*} it follows that $\Ll (\spt,2(\dspt+1)m+1;z) \in \olmodd(\spt,m)$.
 
 If $\ordl(f_2(2(\dspt+1)(m-1)+2;z))=m-1$, then by Lemma~\ref{ceva} we find $h(z) \in \olmodd(\spt,m)$ with
 \begin{equation*}
 f_2 (2(\dspt+1)(m-1)+2;z)  \,|\,  D_1(\ell) \,|\,  X_{1}(\ell)^{\dspt} \equiv h(z)  
\pmod {\ell^m}.
 \end{equation*}
 Hence \begin{align*}
 \Ll (\spt,2(&\dspt+1)m+1;z) =  \Ll (\spt,2(\dspt+1)(m-1)+2;z)  \,|\,  D_1(\ell) \,|\,  X_{1}(\ell)^{\dspt} \\
 & \equiv (f_1(2(\dspt+1)(m-1)+2;z)  \,|\,  D_1(\ell) \,|\,  X_{1}(\ell)^{\dspt}+h(z)) \\
 & \quad+ (f_2(2(\dspt+1)(m-1)+2;z)  \,|\,  D_1(\ell) \,|\,  X_{1}(\ell)^{\dspt}-h(z)) \\
 & \equiv f_1(2(\dspt+1)(m-1)+2;z)  \,|\,  D_1(\ell) \,|\,  X_{1}(\ell)^{\dspt}+h(z) \pmod {\ell^m},
 \end{align*} and therefore $\Ll (\spt,2(\dspt+1)m+1;z) \in \olmeven(\spt,m),$ proving the theorem.
\item 
%
%

If $m=1$, this follows directly from Lemma~\ref{5.3}(2); if $m\ge 2$ this follows directly from the $s=m$ case of Lemma~\ref{5.5}. 
\end{newlist}
\end{proof}

\section{Proofs of main theorems}\label{S6}
\begin{proof}[Proof of Theorems \ref{pr1} and \ref{spt1}]
By the remark after Lemma \ref{bw-3.5}, the stabilized modules $\olmodd(\spt,m)$ and $\olmeven(\spt,m)$ exist.  By the remark following the proof of Theorem \ref{injection}, we have that the ranks of these modules are at most $R_\ell(\spt)$.  Theorem \ref{5.1} gives the bound on $\bl(\spt,m)$ that is stated in Theorem \ref{pr1} and \ref{spt1}. The same reasoning holds for $p_r$.
\end{proof}

\begin{proof}[Proof of Theorems~\ref{hecke} and \ref{hecke-spt}]
In the case of $\spt$, one easily checks that the $\ell$-adic limit of the spaces $\olmodd(\spt,m)$, $\olmeven(\spt, m)$ is in the ordinary part of the space of $\ell$-adic modular forms. For these primes, the dimension of the space is 1. The theorem follows by Theorem 7.1 of~\cite{ono1}.

For $p_r$, the same argument applies when $r$ is odd.  When $r$ is even, the theorem follows by classical facts about $p$-adic modular forms and the ordinary space, since $P_\ell(r, b; 24z)$ has integral weight. 
\end{proof}

\begin{proof}[Proof of Theorems \ref{hecke-cor-pr} and \ref{hecke-cor-spt}]
Let $n \to nc$ in \eqref{hecke-operator} and \eqref{hecke-operator-2}.  The conclusion follows from Theorems \ref{hecke} and \ref{hecke-spt} because $\pa{\frac{nc}{c}} = 0$ and $a(n/c) = 0$ when $n$ is coprime to $c$.  
\end{proof}

\bibliographystyle{plain}
\bibliography{refs}
\end{document}